\documentclass[a4paper,reqno,10pt]{amsart}
\makeatletter
\newcommand*{\rom}[1]{\expandafter\@slowromancap\romannumeral #1@}
\makeatother
\usepackage{epsf,graphicx,epsfig}
\usepackage{amscd}
\usepackage{amsmath,latexsym,amssymb,amsthm}
\usepackage[nospace,noadjust]{cite}
\usepackage{textcomp}
\usepackage{setspace,cite}
\usepackage{mathrsfs,amsmath}
\usepackage{chemarrow}
\usepackage{lscape,fancyhdr,fancybox}
\usepackage{stmaryrd}
\usepackage[all,cmtip]{xy}
\usepackage{tikz-cd}
\usepackage{cancel}
\usetikzlibrary{shapes,arrows,decorations.markings}
\usepackage{graphicx}

\usepackage{color}
\usepackage{bbm, dsfont}
\usepackage{xcolor}
\usepackage{url}
\usepackage{enumerate}
\usepackage[mathscr]{euscript}

  \theoremstyle{plain}
\swapnumbers
    \newtheorem{thm}{Theorem}[section]
    \newtheorem{prop}[thm]{Proposition}

    \newtheorem{corollary}[thm]{Corollary}
    
    \newtheorem{subsec}[thm]{}
\theoremstyle{definition}
    \newtheorem{defn}[thm]{Definition}

    \newtheorem{exam}[thm]{Example}

\theoremstyle{remark}

\title{}
\author{}
\date{}
\usepackage{amssymb}

\usepackage{hyperref}
\hypersetup{
colorlinks,
citecolor=blue,
filecolor=black,
linkcolor=blue,
urlcolor=black
}

\begin{document}

\title[]{Difference $2$-algebras and difference $A_\infty$-algebras}

\author{Apurba Das}
\address{Department of Mathematics,
Indian Institute of Technology, Kharagpur 721302, West Bengal, India.}
\email{apurbadas348@gmail.com, apurbadas348@maths.iitkgp.ac.in}

\begin{abstract}
A difference operator on an associative algebra is an algebraic abstraction of the forward and backward difference operators. In this paper, we first introduce difference operators on associative $2$-algebras and consider the category of difference associative $2$-algebras. Subsequently, we also introduce difference operators on a given $A_\infty$-algebra in terms of their Maurer-Cartan characterization. We prove that the category of difference associative $2$-algebras and the category of $2$-term difference $A_\infty$-algebras are equivalent. We characterize skeletal and strict $2$-term difference $A_\infty$-algebras by respectively third cocycles and crossed modules of difference algebras. Finally, we define the notion of a $2$-term bimodule up to homotopy over a difference algebra, which in turn yields a construction of a $2$-term difference $A_\infty$-algebra.
\end{abstract}

\maketitle





\tableofcontents

\section{Introduction}

\subsection{Differential operators, difference operators and difference algebras}
A differential operator on an associative algebra is a generalization of the first-order differentiation operator on the space of differentiable functions. The origin of differential operators lies in differential equations \cite{ritt}. Later, such operators appear in many areas of mathematics, such as Galois theory, differential geometry and algebraic cohomology theory \cite{kolchin,magid,bott}. Loday studied differential associative algebras from an operadic point of view and showed that the corresponding operad is Koszul \cite{loday}. In \cite{das-mandal,guo-li,laza}, the authors have studied differential associative algebras extensively from a cohomological perspective. As a generalization of differential operators, many people have shown their interest in differential operators of weight $\lambda \in {\bf k}$ \cite{guo-li}. This generalized notion provides a unified approach to study usual differential operators (corresponding to the case $\lambda = 0$) and well-known difference operators (corresponding to the case $\lambda = 1$) that appear in Newton's forward and backward interpolations. Recall that a difference operator on an associative algebra $A$ is a linear map $d: A \rightarrow A$ that satisfies
\begin{align}\label{diffe-iden}
    d(ab) = d(a) b + a d(b) + d(a) d(b), \text{ for } a, b \in A.
\end{align}
A difference operator or a differential operator of weight $1$ is also known as a crossed homomorphism \cite{das}. An associative algebra with a difference operator is often called a difference algebra. In \cite{das,das-ram}, the authors have developed cohomology theories associated with difference operators and difference algebras. It is important to note that differential operators and difference operators have also been considered on Lie algebras and other quadratic algebras \cite{lyu,jia}.

\subsection{Homotopy algebras and 2-algebras} 
In mathematics, higher algebraic structures appear when we increase the flexibility of a given type of algebraic structure. This can often be done in two ways: by homotopification and categorification. Homotopy algebras are exactly the homotopification of algebras. The notion of $A_\infty$-algebras, also called strongly homotopy associative algebras, first appeared in the work of Stasheff in his study of topological loop spaces \cite{stas}. Intuitively, an $A_\infty$-algebra is a generalization of graded associative algebras in which the associativity holds up to homotopy \cite{board}. In its simple form, an $A_\infty$-algebra can be described by a square-zero coderivation on the reduced tensor coalgebra of a graded vector space. Such homotopy algebras are also used to describe minimal models of differential graded associative algebras and D-branes in string theory \cite{keller}. The Lie analogue of $A_\infty$-algebras, which are called $L_\infty$-algebras, has many applications in deformation theory, differential geometry and mathematical physics \cite{marco,sheng-zhu}. See also \cite{keller,loday,lada-markl,lada-s} for more details on $A_\infty$-algebras and $L_\infty$-algebras.

\medskip

On the other hand, $2$-algebras are the categorification of algebras. The concept of Lie $2$-algebras first appeared in the work of Baez and Crans \cite{baez-crans} while categorifying Lie algebras. Lie $2$-algebras are closely related to $L_\infty$-algebras as observed by the same authors. In particular, they showed that the category of Lie $2$-algebras is equivalent to the category of $2$-term $L_\infty$-algebras. Subsequently, categorifications of other algebraic structures, such as groups and associative algebras, are also widely studied. Among others, the authors in \cite{das-hom,kh} have considered associative $2$-algebras as the categorification of associative algebras, and find their relations with $2$-term $A_\infty$-algebras. Recently, various authors have considered Rota-Baxter operators and some other linear operators on homotopy algebras and $2$-algebras. Notably, Zhang and Liu \cite{zhang-liu} extend the result of Baez and Crans by showing that the category of Rota-Baxter Lie $2$-algebras is equivalent to the category of $2$-term Rota-Baxter $L_\infty$-algebras.

\subsection{Layout of the paper} The aim of the present paper is to consider higher structures that are obtained by increasing the flexibility of difference algebras.

\medskip

We begin by considering difference operators on associative $2$-algebras. Unlike difference operators on ordinary associative algebras, here we consider the functors defined by the left and right-hand sides of (\ref{diffe-iden}) to be naturally isomorphic, which satisfy a coherence law (cf. Definition \ref{defi-2-diff}). An associative $2$-algebra endowed with a difference operator is called a difference associative $2$-algebra. It follows that a difference associative $2$-algebra can be viewed as the categorification of a difference algebra. In Theorem \ref{thm-categ-diff2}, we show that the collection of all difference associative $2$-algebras and homomorphisms between them is a category, denoted by {\sf DiffAss2}.

\medskip

To find the homotopification of a difference algebra, we first consider the Maurer-Cartan characterization of difference operators on a given associative algebra. Then we generalize this construction to the context of homotopy algebras. Specifically, given an $A_\infty$-algebra $\mathcal{A}$, we construct an $L_\infty$-algebra through Voronov's derived bracket construction whose Maurer-Cartan elements are termed as difference operators on $\mathcal{A}$. An $A_\infty$-algebra endowed with a difference operator is called a difference $A_\infty$-algebra. It follows that a difference $A_\infty$-algebra can be regarded as the homotopification of a difference algebra. Subsequently, we focus on those difference $A_\infty$-algebras whose underlying graded vector space is concentrated in arities $0$ and $1$. We call them $2$-term difference $A_\infty$-algebras. We observe that the collection of $2$-term difference $A_\infty$-algebras and suitable homomorphisms between them is a category, denoted by ${\sf 2TermDiffA_\infty}$ (cf. Theorem \ref{thm-categ-2term}). As a main result of this paper, we prove that the category ${\sf DiffAss2}$ of difference associative $2$-algebras and the category ${\sf 2TermDiffA_\infty}$ of $2$-term difference $A_\infty$-algebras are equivalent (cf. Theorem \ref{thm-main}).


\medskip

Subsequently, we consider skeletal and strict $2$-term difference $A_\infty$-algebras. We show that skeletal $2$-term difference $A_\infty$-algebras are characterized by third cocycles of difference algebras (cf. Theorem \ref{thm-skeletal}). On the other hand, strict $2$-term difference $A_\infty$-algebras have a 1-1 correspondence with crossed modules of difference algebras (cf. Theorem \ref{thm-strict}). Finally, we define the notion of a $2$-term bimodule up to homotopy over an associative algebra, generalizing the classical concept of representation up to homotopy of Lie algebroids \cite{abad-crainic,sheng-zhu}. We observe that the semidirect product of an associative algebra with a $2$-term bimodule up to homotopy carries a $2$-term $A_\infty$-algebra structure (cf. Theorem \ref{thm-2-hom}). Thereafter, we extend this construction to the context of difference algebras. Namely, we consider a $2$-term bimodule up to homotopy over a difference algebra, and show that the corresponding semidirect product is a $2$-term difference $A_\infty$-algebra (cf. Proposition \ref{prop-2-diff-hom}). Hence, one may also construct a difference associative $2$-algebra.

\subsection{Organization} The paper is organized as follows. In Section \ref{sec2}, we recall some necessary backgrounds on difference algebras, associative $2$-algebras and $2$-term $A_\infty$-algebras. In Section \ref{sec3}, we introduce difference operators on associative $2$-algebras and consider the category ${\sf DiffAss2}$ of difference associative $2$-algebras. In Section \ref{sec4}, we consider difference $A_\infty$-algebras, and show that the category ${\sf DiffAss2}$ and the category ${\sf 2TermDiffA_\infty}$ of $2$-term difference $A_\infty$-algebras are equivalent. Finally, in Section \ref{sec5}, we first provide characterizations of skeletal and strict $2$-term difference $A_\infty$-algebras. Then, we define the notion of a $2$-term bimodule up to homotopy over a difference algebra, and study the corresponding semidirect product.

\medskip

All (graded) vector spaces, (graded) linear maps, multilinear maps and unadorned tensor products are over a field ${\bf k}$ of characteristics $0$ unless specified otherwise.

\section{Background on difference algebras and higher associative algebras}\label{sec2}
In this section, we recall some basics on difference algebras (including their cohomology), associative $2$-algebras and $2$-term $A_\infty$-algebras. We refer the reader to the references \cite{das-hom,das,baez-crans,guo-li,kh} for further details.

\begin{defn}
    (i) Let $A$ be an associative algebra. A linear map $d : A \rightarrow A$ is said to be a {\bf difference operator} (also called a {\bf differential operator of weight 1} or a {\bf crossed homomorphism}) on $A$ if
    \begin{align*}
        d (a b) = d(a) b + a d(b) + d(a) d(b), \text{ for } a, b \in A.
    \end{align*}
    A pair $(A, d)$ consisting of an associative algebra $A$ endowed with a distinguished difference operator $d: A \rightarrow A$ is called a {\bf difference algebra}.

    \medskip

    (ii) Let $(A, d)$ and $(A', d')$  be two difference algebras. Then a {\bf homomorphism} of difference algebras from $(A, d)$ to $(A', d')$ is an associative algebra homomorphism $\varphi : A \rightarrow A'$ satisfying $d' \circ \varphi = \varphi \circ d$.
\end{defn}

Let $(A, d)$ be a difference algebra. Then a {\bf bimodule} over $(A, d)$ is a pair $(M, \Delta)$, where $M$ is an $A$-bimodule (in which both the left and right $A$-actions on $M$ are denoted simply by the juxtaposition) and $\Delta : M \rightarrow M$ is a linear map such that for any $a \in A$, $u \in M$,
\begin{align*}
    \Delta (au) = d(a) u + a \Delta (u) + d(a) \Delta (u) \quad \text{ and } \quad \Delta (ua) = \Delta (u) a + u d(a) + \Delta (u) d(a).
\end{align*}

\begin{prop}
    Let $(A, d)$ be a difference algebra and $M$ be an $A$-bimodule. Then $M$ has a new $A$-bimodule structure with respect to the new left and right $A$-actions defined by
    \begin{align*}
        A \times M \rightarrow M, ~(a, u) \mapsto (a + d(a)) u ~~~ \text{ and } ~~~ M \times A \rightarrow M, ~ (u, a) \mapsto u (a + d(a)), \text{ for } a \in A, u \in M.
    \end{align*}
(We denote this new $A$-bimodule structure on $M$ by the notation $M^d$.)
\end{prop}

Let $(A, d)$ be a difference algebra and $(M, \Delta)$ be a bimodule over it. For each $n \geq 0$, we set
\begin{align*}
    C^n_\mathrm{Diff} ((A, d); (M, \Delta) ) := \begin{cases}
        M & \text{ if } n =0,\\
        \mathrm{Hom}(A, M) & \text{ if } n = 1, \\
        \mathrm{Hom} (A^{\otimes n} , M) \oplus \mathrm{Hom}  (A^{\otimes (n-1)} , M) & \text{ if } n \geq 2.
    \end{cases}
\end{align*}
Then there is a map $\delta_\mathrm{Diff} :  C^n_\mathrm{Diff} ((A, d); (M, \Delta) ) \rightarrow  C^{n+1}_\mathrm{Diff} ((A, d); (M, \Delta) )$ defined by
\begin{align*}
    &\delta_\mathrm{Diff} (u) := \delta_\mathrm{Hoch} (u), \quad \delta_\mathrm{Diff} (g) := (\delta_\mathrm{Hoch} (g) ~ \! , ~ \! - \partial^{d, \Delta} ( g ) ), \\
    & \quad \delta_\mathrm{Diff} (f, \chi) := \big( \delta_\mathrm{Hoch} (f) ~ \! , ~ \! \delta_\mathrm{Hoch}^d (\chi) + (-1)^n ~ \! \partial^{d, \Delta} (f)   \big),
\end{align*}
for $u \in M$, $g \in \mathrm{Hom}(A, M)$ and $(f, \chi) \in \mathrm{Hom} (A^{\otimes n} , M) \oplus \mathrm{Hom}  (A^{\otimes (n-1)} , M)$. Here $\delta_\mathrm{Hoch}$ (resp. $\delta_\mathrm{Hoch}^d$) is the Hochschild coboundary operator of the associative algebra $A$ with coefficients in the $A$-bimodule $M$ (resp. $M^d$), and the map $\partial^{d, \Delta} : \mathrm{Hom} (A^{\otimes n} , M) \rightarrow \mathrm{Hom} (A^{\otimes n} , M)$ is given by
\begin{align*}
    (\partial^{d, \Delta} f) (a_1, \ldots, a_n) = \sum_{p=1}^n \sum_{1 \leq i_1 < \cdots < i_p \leq n} f \big(   a_1, \ldots, d (a_{i_1}), \ldots, d (a_{i_p}), \ldots, a_n  \big) - \Delta ( f (a_1, \ldots, a_n)),
\end{align*}
for $f \in  \mathrm{Hom} (A^{\otimes n} , M) $ and $a_1, \ldots, a_n \in A$. Then it turns out that $\{ C^\bullet_\mathrm{Diff} ((A, d); (M, \Delta) ), \delta_\mathrm{Diff} \}$ is a cochain complex \cite{guo-li}. The corresponding cohomology is said to be the {\bf cohomology of the difference algebra} $(A, d)$ with coefficients in the bimodule $(M, \Delta)$. We denote the cohomology groups by $H^\bullet_\mathrm{Diff} ((A, d); (M, \Delta) ).$

\medskip

\medskip

A {\bf 2-vector space} is a category internal to the category of vector spaces \cite{baez-crans}. Explicitly, a $2$-vector space is a category $C$ with a vector space of objects $C_0$ and a vector space of morphisms $C_1$ such that the source and target maps $s, t: C_1 \rightarrow C_0$, the identity-assigning map $i: C_0 \rightarrow C_1$, and the composition map $\circ: C_1 \times_{C_0} C_1 \rightarrow C_1$ are all linear maps. We denote a $2$-vector space $C$ as above simply by $C= (C_1 \rightrightarrows C_0)$ when all the structure maps are understood. Given a morphism $f \in C_1$, we often consider its {\em arrow part} $\overrightarrow{f}$ defined by
\begin{align*}
    \overrightarrow{f} := f - i_{s (f)}.
\end{align*}
Let $C= (C_1 \rightrightarrows C_0)$ and $C'= (C'_1 \rightrightarrows C'_0)$ be two $2$-vector spaces. Then a {\bf linear functor} from $C$ to $C'$ is given by a pair $(F_0, F_1)$ of linear maps $F_0 : C_0 \rightarrow C_0'$ and $F_1 : C_1 \rightarrow C_1'$ compatible with the structure maps of $C$ and $D$. The collection of all $2$-vector spaces and linear functors between them is a category.

\begin{defn}
    (i) An {\bf associative $2$-algebra} is a triple $(C,  \bullet  , \mathcal{A})$ consisting of a $2$-vector space $C = (C_1 \rightrightarrows C_0)$ equipped with a bilinear functor $\bullet : C \times C \rightarrow C$ and a trilinear natural isomorphism (called the {\em associator})
    \begin{align*}
        \mathcal{A}_{x, y, z}: (x \bullet y) \bullet z \rightarrow x \bullet (y \bullet z)
    \end{align*}
that makes the following diagram commutative
\[
\xymatrix{
 & (( x \bullet y) \bullet z) \bullet t \ar[ld]_{ \mathcal{A}_{x, y, z} \bullet 1} \ar[rd]^{1} & \\
(x \bullet (y \bullet z) ) \bullet t \ar[d]_{ \mathcal{A}_{x, y \bullet z, t}} &  & ( (x \bullet y) \bullet z ) \bullet t \ar[d]^{\mathcal{A}_{x \bullet y, z, t}} \\
x \bullet ( (y \bullet z) \bullet t )   \ar[rd]_{1 \bullet \mathcal{A}_{y, z, t}} &  & (x \bullet y) \bullet (z \bullet t)  \ar[ld]^{\mathcal{A}_{x, y, z \bullet t}} \\
  & x \bullet (y \bullet (z \bullet t)). & 
}
\]

\medskip

(ii) Let $(C, \bullet, \mathcal{A})$ and $(C', \bullet', \mathcal{A}')$ be two associative $2$-algebras. A {\bf homomorphism} of associative $2$-algebras from $C$ to $C'$ is a linear functor $(F_0, F_1) : C \rightarrow C'$ among the underlying $2$-vector spaces equipped with a bilinear natural isomorphism
\begin{align*}
    F_2 (x, y) : F_0 (x) \bullet' F_0 (y) \rightarrow F_0 (x \bullet y)
\end{align*}
which makes the following diagram commutative
\[
\xymatrix{
( F_0 (x) \bullet' F_0 (y) ) \bullet' F_0 (z) \ar[rr]^{\mathcal{A}'_{ F_0 (x), F_0 (y), F_0 (z)}}  \ar[d]_{F_2 (x, y) \bullet' 1}  &  & F_0 (x) \bullet' (F_0 (y) \bullet' F_0 (z))  \ar[d]^{1 \bullet' F_2 (y, z)}  \\
F_0 (x \bullet y) \bullet' F_0 (z) \ar[d]_{F_2 (x \bullet y, z)} & & F_0 (x) \bullet' F_0 (y \bullet z)  \ar[d]^{F_2 (x, y \bullet z)} \\
F_0 ( (x \bullet y) \bullet z) \ar[rr]_{F_1 (\mathcal{A}_{x, y, z})} &  & F_0 (x \bullet (y \bullet z)).
}
\]
We denote a homomorphism as above by the triple $(F_0, F_1, F_2)$. The collection of all associative $2$-algebras and homomorphisms between them is a category \cite{das-hom,kh}. We denote this category by {\sf Ass2}. It is important to remark that associative $2$-algebras are closely related to those $A_\infty$-algebras whose underlying graded vector space is concentrated in arities $0$ and $1$.
\end{defn}

\begin{defn}\label{defi-2term-ainf}
    (i) A {\bf $2$-term $A_\infty$-algebra} is a triple $\mathfrak{a} = (A_1 \xrightarrow{ \delta } A_0, \odot, \mu)$ consisting of a $2$-term complex $A_1 \xrightarrow{ \delta } A_0$ equipped with a bilinear multiplication $\odot : A_i \times A_j \rightarrow A_{i+j}$ (for $0 \leq i+j \leq 1$) and a trilinear operation $\mu : A_0 \times A_0 \times A_0 \rightarrow A_1$ such that for any $x, y, z, t \in A_0$ and $h, k \in A_1$, the following list of identities hold:
\begin{align}
    \delta (x \odot h) =~& x \odot \delta (h), \tag{A1} \label{a1}\\
    \delta (h \odot x) =~& \delta (h) \odot x, \tag{A2} \label{a2}\\
    \delta (h) \odot k =~& h \odot \delta (k), \tag{A3} \label{a3}\\
    x \odot (y \odot z) - (x \odot y) \odot z =~& \delta (\mu (x, y, z)), \tag{A4} \label{a4}\\
    x \odot (y \odot h) - (x \odot y) \odot h =~& \mu (x, y, \delta (h)), \tag{A5} \label{a5}\\
    x \odot (h \odot y) - (x \odot h) \odot y =~& \mu (x, \delta (h), y), \tag{A6} \label{a6}\\
    h \odot (x \odot y) - (h \odot x) \odot y =~& \mu (\delta (h), x, y), \tag{A7} \label{a7}\\
    x \odot \mu (y, z, t) - \mu (x \odot y, z, t) + \mu (x, y \odot z, t) &- \mu (x, y, z \odot t) + \mu (x, y, z) \odot t = 0. \tag{A8} \label{a8}
\end{align}

    (ii) Let  $\mathfrak{a} = (A_1 \xrightarrow{ \delta } A_0, \odot, \mu)$ and  $\mathfrak{a}' = (A'_1 \xrightarrow{ \delta' } A'_0, \odot', \mu')$ be $2$-term $A_\infty$-algebras. Then an {\bf $A_\infty$-homomorphism} from $\mathfrak{a}$ to $\mathfrak{a}'$ is a triple $(\varphi_0, \varphi_1, \varphi_2)$, where $\varphi_0 : A_0 \rightarrow A_0'$ and $\varphi_1 : A_1 \rightarrow A_1'$ are linear maps satisfying $\varphi_0 \circ \delta = \delta' \circ \varphi_1$, and $\varphi_2 : A_0 \times A_0 \rightarrow A_1'$ is a bilinear map such that for all $x, y, z \in A_0$ and $h \in A_1$, the following list of identities hold:
    \begin{align}
         \varphi_0 (x \odot y) - \varphi_0 (x) \odot' \varphi_0 (y) =~& \delta' (\varphi_2 (x, y)), \\
         \varphi_1 (x \odot h) - \varphi_0 (x) \odot' \varphi_1 (h) =~& \varphi_2 (x, \delta (h)), \\
         \varphi_1 (h \odot x) - \varphi_1 (h) \odot' \varphi_0 (x) =~& \varphi_2 (\delta (h), x), \\
         \varphi_1 (\mu (x, y, z)) - \mu' (\varphi_0 (x), \varphi_0 (y), \varphi_0 (z)) =~& \varphi_2 (x, y \odot z) - \varphi_2 (x \odot y, z) \nonumber \\
         &+ \varphi_0 (x) \odot' \varphi_2 (y, z) - \varphi_2 (x, y)\odot' \varphi_0 (z).
    \end{align}
\end{defn}
The collection of all $2$-term $A_\infty$-algebras and $A_\infty$-homomorphisms between them is also a category \cite{keller}. This category is denoted by ${\sf 2TermA_\infty}.$ It has been observed that the categories {\sf Ass2} and ${\sf 2TermA_\infty}$ are equivalent \cite{das-hom,kh}. We will implicitly use this result to prove Theorem \ref{thm-main} in Section \ref{sec4}.

\section{Difference associative 2-algebras}\label{sec3}

In this section, we first introduce difference operators on associative $2$-algebras, and consider the category {\sf DiffAss2} of difference associative $2$-algebras. We also observe that the category {\sf DiffAss2} has a subcategory {\sf SDiffAss2} consisting of all strict difference associative $2$-algebras.

\begin{defn}\label{defi-2-diff}
    Let $(C, \bullet, \mathcal{A})$ be an associative $2$-algebra. A {\bf difference operator} on $C$ is a linear functor $D = (D_0, D_1) : C \rightarrow C$ together with a bilinear natural isomorphism
    \begin{align}\label{defi-D}
        \mathcal{D}_{x, y} : D_0 (x \bullet y) \rightarrow D_0 (x) \bullet y + x \bullet D_0 (y) + D_0 (x) \bullet D_0 (y)
    \end{align}
    making the following diagram commutative:
    \begin{align}\label{2-diff}
    \xymatrix{
     & D_0 ( (x \bullet y) \bullet z) \ar[ld]_{\mathcal{D}_{x \bullet y, z}} \ar[rd]^{D_1 (\mathcal{A}_{x, y, z})} & \\
 \substack{ D_0 (x \bullet y) \bullet z + (x \bullet y) \bullet D_0 (z) \\ + D_0 (x \bullet y) \bullet D_0 (z)} \ar[d]_{ \mathcal{D}_{x, y} \bullet 1 ~+~ 1 ~+~ \mathcal{D}_{x, y} \bullet 1} & &  D_0 (x \bullet (y \bullet z)) \ar[d]^{ \mathcal{D}_{x, y \bullet z} } \\
   \substack{  (D_0 (x) \bullet y) \bullet z + (x \bullet D_0 (y)) \bullet z + (D_0 (x) \bullet D_0 (y)) \bullet z \\
   + (x \bullet y) \bullet D_0 (z) + (D_0 (x) \bullet y) \bullet D_0 (z) \\
   + (x \bullet D_0 (y)) \bullet D_0 (z) + (D_0 (x) \bullet D_0 (y)) \bullet D_0 (z) } \ar[rd]_{ \substack{ \mathcal{A}_{D_0 (x), y, z} +  \mathcal{A}_{x, D_0 (y), z} + \mathcal{A}_{D_0 (x), D_0 (y) , z} \\  +  \mathcal{A}_{x, y, D_0 (z)} +  \mathcal{A}_{D_0 (x), y ,  D_0 (z)}  \\ +  \mathcal{A}_{x, D_0 (y), D_0 (z) } +  \mathcal{A}_{D_0 (x) , D_0 (y), D_0 (z) } } } & &
   \substack{ D_0 (x) \bullet (y \bullet z) + x \bullet D_0 (y \bullet z) \\ + D_0 (x) \bullet D_0 (y \bullet z)} \ar[ld]^{\quad 1 ~+~ 1 \bullet \mathcal{D}_{y, z} ~+~ 1 \bullet \mathcal{D}_{y, z}} \\
     & P &
    }
    \end{align}
    where
    \begin{align*}
    P =~&  D_0 (x) \bullet (y \bullet z) + x \bullet (D_0 (y) \bullet z) + x \bullet (y \bullet D_0 (z)) 
     + x \bullet (D_0 (y) \bullet D_0 (z)) + D_0 (x) \bullet (D_0 (y) \bullet z) \\
     &+ D_0 (x) \bullet (y \bullet D_0 (z)) + D_0 (x) \bullet (D_0 (y) \bullet D_0 (z)). 
     \end{align*}
\end{defn}
\noindent A tuple $((C, \bullet, \mathcal{A}), D, \mathcal{D})$ consisting of an associative $2$-algebra together with a difference operator is called a {\bf difference associative $2$-algebra.} We shall often denote a difference associative $2$-algebra as above by $(C, D, \mathcal{D})$ if the underlying associative $2$-algebra structure on $C$ is clear.

\begin{defn}\label{defi-da2-hom}
    Let $((C, \bullet, \mathcal{A}), D, \mathcal{D})$ and $((C', \bullet', \mathcal{A}'), D', \mathcal{D}')$ be two difference associative $2$-algebras. Then a {\bf homomorphism of difference associative $2$-algebras} from the former one to the latter one is a homomorphism $(F_0, F_1, F_2): (C, \bullet, \mathcal{A}) \rightarrow (C', \bullet', \mathcal{A}')$ of the underlying associative $2$-algebras with a natural linear isomorphism
\begin{align}\label{f3}
F_3 (x) : D_0' (F_0 (x)) \rightarrow F_0 (D_0 (x))
\end{align}
making the following diagram commutative:
    \begin{align}\label{2-diff-homo}
    \xymatrix{
 D_0' ( F_0 (x) \bullet' F_0 (y) )  \ar[d]_{D_1' F_2 (x, y)}  \ar[rrr]^{\mathcal{D}'_{F_0 (x), F_0 (y)}} & & & \substack{ D_0' (F_0 (x)) \bullet' F_0 (y) + F_0 (x) \bullet' D_0' (F_0 (y)) \\ + D_0' (F_0 (x)) \bullet' D_0' (F_0 (y)) } \ar[d]^{ \substack{F_3 (x) \bullet' 1 + 1 \bullet' F_3 (y)  \\ + F_3 (x) \bullet' F_3 (y) } }  \\
  D_0' (F_0 (x \bullet y)) \ar[d]_{F_3 (x \bullet y)}  & & & \substack{ F_0 (D_0 (x)) \bullet' F_0 (y) + F_0 (x) \bullet' F_0 (D_0 (y)) \\ + F_0 (D_0 (x)) \bullet' F_0 (D_0 (y))} \ar[d]^{  \substack{  F_2 (D_0 (x), y) + F_2 (x, D_0 (y)) \\ + F_2 (D_0 (x), D_0 (y))} } \\
  F_0 (D_0 (x \bullet y))  \ar[rrr]_{F_1 (\mathcal{D}_{x, y})} &  & & \substack{ F_0 (D_0 (x) \bullet y) + F_0 (x \bullet D_0 (y)) \\ + F_0 ( D_0 (x) \bullet D_0 (y)).}
    }
    \end{align}
    A homomorphism of difference associative $2$-algebras as above is denoted by the tuple $F = (F_0, F_1, F_2, F_3)$.
\end{defn}

With the above definitions, we have the following result.

\begin{thm}\label{thm-categ-diff2}
    The collection of difference associative $2$-algebras as objects and homomorphisms between them is a category. (We denote this category by {\sf DiffAss2}).
\end{thm}

\begin{proof}
    Here, we only define the composition of homomorphisms and describe identity homomorphisms. Let $F = (F_0, F_1, F_2, F_3) : (C, D, \mathcal{D}) \rightarrow (C', D', \mathcal{D}')$ and $G = (G_0, G_1, G_2, G_3) : (C', D', \mathcal{D}') \rightarrow (C'', D'', \mathcal{D}'')$ be two homomorphisms between difference associative $2$-algebras. Then their composition 
    \begin{align*}
    G \circ F : (C, D, \mathcal{D}) \rightarrow (C', D', \mathcal{D}'')
    \end{align*}
    is defined to be the usual composition of the underlying linear functors (i.e. $(G \circ F)_0 = G_0 \circ F_0$ and $(G \circ F)_1 = G_1 \circ F_1$), and by letting $(G \circ F)_2$ and $(G \circ F)_3$ be the following compositions
    \begin{center}
        $(G \circ F)_2 (x, y) : (G \circ F)_0 (x) \bullet'' (G \circ F)_0 (y) \xrightarrow{ G_2 (F_0 (x), F_0 (y))} G_0 ( F_0 (x) \bullet'F_0 (y) ) \xrightarrow{ G_1 (F_2 (x, y) )} (G \circ F)_0 (x \bullet y),$ 
        \end{center}
        \begin{center}
       $ (G \circ F)_3 (x) : D_0'' ( (G \circ F)_0 (x)) \xrightarrow{ G_3 (F_0 (x))} G_0 (D_0' (F_0 (x))) \xrightarrow{ G_1 (F_3 (x))} (G \circ F)_0 (D_0 (x)).$
    \end{center}
    On the other hand, for any difference associative $2$-algebra $(C, D, \mathcal{D})$, we define the identity homomorphism $1_{ (C, D, \mathcal{D})} : (C, D, \mathcal{D}) \rightarrow (C, D, \mathcal{D})$ by taking the identity linear functor as its underlying functor, together with the identity natural transformations  $(1_{ (C, D, \mathcal{D})} )_2$ and  $(1_{ (C, D, \mathcal{D})} )_3$. With these structures, it is easy to see that {\sf DiffAss2} is a category.
\end{proof}

Let $( C, D, \mathcal{D})$ be a difference associative $2$-algebra. It is said to be {\bf strict} if
\begin{itemize}
    \item[(i)] the underlying associative $2$-algebra $C = (C, \bullet, \mathcal{A})$ is {\em strict} in the sense that the associator $\mathcal{A}$ is the identity isomorphism,
    \item[(ii)] $\mathcal{D}$ is also identity isomorphism.
\end{itemize}
The collection of all strict difference associative $2$-algebras and homomorphisms of difference associative $2$-algebras is a category, denoted by {\sf SDiffAss2}. It is easy to see that {\sf SDiffAss2} is a subcategory of {\sf DiffAss2}.

\section{Difference A$_\infty$-algebras}\label{sec4}

In this section, we first recall the Maurer-Cartan characterization of difference operators on a given associative algebra. Subsequently, we generalize this construction to obtain the definition of a difference operator on an $A_\infty$-algebra. An $A_\infty$-algebra with a distinguished difference operator is called a difference $A_\infty$-algebra. We explicitly describe $2$-term difference $A_\infty$-algebras, and show that the category of $2$-term difference $A_\infty$-algebras is equivalent to the category {\sf DiffAss2} of difference associative $2$-algebras considered in the previous section.

\medskip

We shall begin with $L_\infty$-algebras and their Maurer-Cartan elements \cite{lada-markl,lada-s,getzler}. Note that $L_\infty$-algebras considered below are different from the traditional definition \cite{lada-markl}. However, one is related to the other by a degree shift.

\begin{defn}
    An {\bf $L_\infty$-algebra} is a pair $(\mathfrak{L}, \{ l_k \}_{k=1}^\infty)$ of a graded vector space $\mathfrak{L} = \oplus_{i \in \mathbb{Z}} \mathfrak{L}_i$ together with a collection $\{ l_k : \mathfrak{L}^{\otimes k} \rightarrow \mathfrak{L} \}_{k=1}^\infty$ of degree $-1$ graded symmetric linear maps satisfying 
    \begin{align*}
        \sum_{i+j = n+1} \sum_{\sigma \in \mathbb{S}_{(i, n-i)}} \epsilon (\sigma) ~ \! l_j \big(  l_i ( x_{\sigma (1)}, \ldots, x_{\sigma (i)}) , x_{\sigma (i+1)}, \ldots, x_{\sigma (n)} \big) = 0,
    \end{align*}
    for each $n \geq 1$ and homogeneous elements $x_1, \ldots, x_n \in \mathfrak{L}$. Here $\epsilon (\sigma) = \epsilon (\sigma; x_1, \ldots, x_n)$ is the {\em Koszul sign} that appears in the graded context while interchanging the positions of graded objects.
\end{defn}

All $L_\infty$-algebras considered in this paper are {\em weakly filtered} in the sense of \cite{getzler}. This, in turn, implies that the infinite sum considered in the following definition is always convergent.

\begin{defn}
    Let $(\mathfrak{L}, \{ l_k \}_{k=1}^\infty)$ be an $L_\infty$-algebra. Then an element $\alpha \in \mathfrak{L}_0$ is said to be a {\bf Maurer-Cartan element} of the $L_\infty$-algebra if $\alpha$ satisfies
    \begin{align*}
        \sum_{k=1}^\infty \frac{1}{k! } ~ \! l_k (\alpha, \ldots, \alpha) = 0.
    \end{align*}
\end{defn}

Voronov \cite{voro} gives a useful construction of an $L_\infty$-algebra. First, we recall that a {\em $V$-data} is a quadruple $(\mathcal{G}, \mathcal{H}, P, \pi)$ consisting of a graded Lie algebra $\mathcal{G}$ (with the graded Lie bracket $[~, ~]$), an abelian graded Lie subalgebra $\mathcal{H} \subset \mathcal{G}$, a projection map $P : \mathcal{G} \rightarrow \mathcal{H}$ whose kernel $\mathrm{ker} (P) \subset \mathcal{G}$ is a graded Lie subalgebra, and an element $\pi \in \mathrm{ker} (P)_{- 1}$ that satisfies $[\pi, \pi] = 0$. With this notation, Voronov's construction is given by the following.

\begin{thm}\label{thm-voro}
    Let $(\mathcal{G}, \mathcal{H}, P, \pi)$ be a $V$-data. Then the pair $(\mathcal{H}, \{ l_k \}_{k=1}^\infty)$ is an $L_\infty$-algebra, where 
    \begin{align*}
        l_k (a_1, \ldots, a_k) := P [ \cdots [[ \pi, a_1], a_2], \ldots, a_k],
    \end{align*}
    for all $k \geq 1$ and homogeneous elements $a_1, \ldots, a_k \in \mathcal{H}$.
\end{thm}

Let $A$ be an associative algebra, and $\overline{A}$ be another copy of the vector space $A$. Given any $a \in A$, we denote the corresponding element in $\overline{A}$ simply by $\overline{a}$. Thus, $a$ and $\overline{a}$ represent the same element. Then it is easy to see that the vector space $ \overline{A} \oplus A$ carries a new associative algebra structure with the multiplication given by 
\begin{align*}
     (\overline{a}, x) \ast (\overline{b}, y) = ( \overline{a b}, a y + x b + xy),
\end{align*}
for $(\overline{a}, x) ,  (\overline{b}, y) \in \overline{A} \oplus A$. We note that a linear map $d : A \rightarrow A$ is a difference operator on $A$ if and only if the graph of the corresponding map (also denoted by the same notation) $d: \overline{A} \rightarrow A$, $\overline{a} \mapsto d(a)$, i.e., 
\begin{align*}
    \mathrm{Graph} ~ \! (d) := \{ ( \overline{a}, d(a)) ~ \! | ~ \! \overline{a} \in \overline{A} \}
\end{align*}
is a subalgebra of the associative algebra $(  \overline{A} \oplus A, \ast )$ considered above. To obtain another characterization of difference operators, we first consider Gerstenhaber's graded Lie algebra
\begin{align*}
    \mathcal{G} = \big( \oplus_{n= - \infty}^0 \mathrm{Hom} ( (\overline{A} \oplus A)^{\otimes (1-n)}, \overline{A} \oplus A ) , [~, ~] \big)
\end{align*}
on the space of all multilinear maps on the vector space $\overline{A} \oplus A$. It is easy to see that the graded space $\mathcal{H} =  \oplus_{n= - \infty}^0 \mathrm{Hom} ( \overline{A}^{\otimes (1-n)}, A )$ is an abelian graded Lie subalgebra of $\mathcal{G}$. Let $ P: \mathcal{G} \rightarrow \mathcal{H}$ be the projection map onto the subspace $\mathcal{H}$. Then $\mathrm{ker} (P)$ is also a graded Lie subalgebra of $\mathcal{G}$. On the other hand, since $(\overline{A} \oplus A, \ast)$ is an associative algebra structure on the vector space $\overline{A} \oplus A$, the multiplication $\ast$ corresponds to an element $\pi \in \mathrm{Hom} ( (\overline{A} \oplus A)^{\otimes 2}, \overline{A} \oplus A )$ such that 
\begin{align*}
    \pi ( (\overline{a}, x), (\overline{b}, y) ) =  (\overline{a}, x) \ast (\overline{b}, y), \text{ for all } (\overline{a}, x), (\overline{b}, y) \in \overline{A} \oplus A.
\end{align*}
Then $\pi \in \mathrm{ker}(P)_{-1}$ and satisfies $[\pi, \pi] = 0$. Combining all these, we obtain a $V$-data $(\mathcal{G}, \mathcal{H}, P, \pi)$. Hence, by applying Theorem \ref{thm-voro}, we get the following \cite{das,das-ram}.

\begin{thm}\label{mc-gla}
   Let $A$ be an associative algebra. Then $ (\mathcal{H} =  \oplus_{n= - \infty}^0 \mathrm{Hom} ( \overline{A}^{\otimes (1-n)}, A ) , \{ l_k \}_{k=1}^\infty)$ is an $L_\infty$-algebra, where
    \begin{align*}
        l_1 (f_1) := P [\pi, f_1], \qquad
        l_2 (f_1, f_2) := P [[ \pi, f_1], f_2] ~ \quad \text{and} \quad
        l_k =0, \text{ for } k \geq 3.
    \end{align*}
    Moreover, a linear map $d: A \rightarrow A$ is a difference operator on $A$ if and only if the corresponding map $d: \overline{A} \rightarrow A$ (considered as an element of $\mathcal{H}_0 = \mathrm{Hom} (\overline{A}, A)$) is a Maurer-Cartan element of the $L_\infty$-algebra $ (\mathcal{H} , \{ l_k \}_{k=1}^\infty)$.
\end{thm}

The $L_\infty$-algebra obtained in the above theorem is indeed a differential graded Lie algebra (after a degree shift), as only the maps $l_1$ and $l_2$ are nontrivial. In what follows, we generalise the above characterization to the homotopy context. First, recall that an {\bf $A_\infty$-algebra} is a pair $(\mathfrak{A}, \{ \mu_k \}_{k=1}^\infty)$ consisting of a graded vector space $\mathfrak{A} = \oplus_{i \in \mathbb{Z}} \mathfrak{A}_i$ together with a collection $\{ \mu_k : \mathfrak{A}^{\otimes k} \rightarrow \mathfrak{A} \}_{k=1}^\infty$ of degree $-1$ graded linear maps satisfying 
    \begin{align*}
        \sum_{k+l = n+1} \sum_{i=1}^{n-l+1} (-1)^{ |a_1| + \cdots + |a_{i-1}|} ~ \! \mu_k \big(  a_1, \ldots, a_{i-1} , \mu_l (a_i, \ldots, a_{i+l-1}), a_{i+l} , \ldots, a_n \big) =0,
    \end{align*}
    for each $n \geq 1$ and homogeneous elements $a_1, \ldots, a_n \in \mathfrak{A}$. This definition is also different from the traditional definition of an $A_\infty$-algebra \cite{keller,lada-markl}; however, one is related to the other by a degree shift.

    \medskip

Let $\mathfrak{A} = \oplus_{i \in \mathbb{Z}} \mathfrak{A}_i$ be a graded vector space (not necessarily having any additional structure). Consider the free reduced tensor algebra $\overline{T} (\mathfrak{A}) = \oplus_{n=1}^\infty \mathfrak{A}^{\otimes n}$ over the graded vector space $\mathfrak{A}$. For each $n \in \mathbb{Z}$, let $C^n (\mathfrak{A}, \mathfrak{A}) := \mathrm{Hom}^n ( \overline{T} (\mathfrak{A}), \mathfrak{A})$ be the space of all degree $n$ graded linear maps from the graded vector space $\overline{T} (\mathfrak{A})$ to $\mathfrak{A}$. Hence, an element $\mu \in C^n (\mathfrak{A}, \mathfrak{A})$ is given by a formal sum $\mu = \sum_{k=1}^\infty \mu_k$, where $\mu_k : \mathfrak{A}^{\otimes k} \rightarrow \mathfrak{A}$ is a degree $n$ graded linear map. Given any $\mu = \sum_{k=1}^\infty \mu_k \in C^n (\mathfrak{A}, \mathfrak{A})$ and $\nu = \sum_{l=1}^\infty \nu_l \in C^m (\mathfrak{A}, \mathfrak{A})$, we define a bracket $\llbracket \mu, \nu \rrbracket \in C^{m+n} (\mathfrak{A}, \mathfrak{A})$ by
\begin{align*}
    \llbracket \mu, \nu \rrbracket = \sum_{p =1}^\infty \big(  \sum_{k+l = p+1} \mu_k \diamond \nu_l - (-1)^{mn} ~ \! \nu_l \diamond \mu_k    \big),
\end{align*}
where
\begin{align*}
    (\mu_k \diamond \nu_l) (a_1, \ldots, a_p) := \sum_{i=1}^{p-l+1} (-1)^{|a_1| + \cdots + |a_{i-1}|} ~\! \mu_k \big( a_1, \ldots, a_{i-1}, \nu_l (a_i, \ldots, a_{i+l-1}), a_{i+l}, \ldots, a_p \big).
\end{align*}
Then it turns out that the graded vector space $\oplus_{n \in \mathbb{Z}} C^n (\mathfrak{A}, \mathfrak{A})$ with the above-defined bracket is a graded Lie algebra. Moreover, an element $\mu = \sum_{k=1}^\infty \mu_k \in C^{-1} (\mathfrak{A}, \mathfrak{A})$ is a Maurer-Cartan element of the graded Lie algebra $\big(   \oplus_{n \in \mathbb{Z}} C^n (\mathfrak{A}, \mathfrak{A}), [~, ~]  \big)$, i.e., $[\mu, \mu] = 0$ if and only if $(\mathfrak{A}, \{ \mu_k \}_{k=1}^\infty)$ is an $A_\infty$-algebra.

\medskip

 Let $(\mathfrak{A}, \{ \mu_k \}_{k=1}^\infty)$ be an $A_\infty$-algebra. As before, let $\overline{\mathfrak{A}}$ be another copy of $\mathfrak{A}$. Then we endow $\overline{\mathfrak{A}} \oplus \mathfrak{A}$ an $A_\infty$-algebra structure whose structure maps $ \{  \widetilde{\mu_k} : (\overline{\mathfrak{A}} \oplus \mathfrak{A})^{\otimes k} \rightarrow \overline{\mathfrak{A}} \oplus \mathfrak{A} \}_{k=1}^\infty$ are given by 
\begin{align}\label{semid-inf}
     \widetilde{\mu_k}  (( \overline{a_1}, x_1), \ldots, ( \overline{a_k}, x_k) ) = \big( ~ \!  \overline{\mu_k (a_1, \ldots, a_k)} ~ \! , ~ \! \sum_{p=1}^k \sum_{1 \leq i_1 < \cdots < i_p \leq k} \mu_k ( a_1, \ldots, x_{i_1}, \ldots, x_{i_p}, \ldots, a_k ) \big),
\end{align}
for any $k \geq 1$ and homogeneous elements $( \overline{a_1}, x_1), \ldots, ( \overline{a_k}, x_k) \in \overline{\mathfrak{A}} \oplus \mathfrak{A}$. We now consider the graded Lie algebra 
\begin{align*}
    \mathcal{G} = \big(    \oplus_{n \in \mathbb{Z}} C^n ( \overline{\mathfrak{A}} \oplus \mathfrak{A}, \overline{\mathfrak{A}} \oplus \mathfrak{A}   ), [~, ~] \big)
\end{align*}
associated to the graded vector space $\overline{\mathfrak{A}} \oplus \mathfrak{A}$. Then it is easy to see that $\mathcal{H} =  \oplus_{n \in \mathbb{Z}} C^n ( \overline{\mathfrak{A}} ,  \mathfrak{A}   ) $ is an abelian graded Lie subalgebra of $\mathcal{G}$. Let $P: \mathcal{G} \rightarrow \mathcal{H}$ be the projection onto the graded subspace $\mathcal{H}$. Then $\mathrm{ker} (P)$ is a graded Lie subalgebra of $\mathcal{G}$. Finally, we observe that the $A_\infty$-algebra structure $( \overline{\mathfrak{A}} \oplus \mathfrak{A}, \{ \widetilde{\mu_k} \}_{k=1}^\infty)$ considered in (\ref{semid-inf}) gives rise to a Maurer-Cartan element $\pi = \sum_{k=1}^\infty \widetilde{\mu_k} \in  C^{-1} ( \overline{\mathfrak{A}} \oplus \mathfrak{A}, \overline{\mathfrak{A}} \oplus \mathfrak{A} )$ of the graded Lie algebra $\mathcal{G}$. Additionally, we also have $\pi \in \mathrm{ker} (P)_{-1}$. Hence, we obtain a $V$-data $(\mathcal{G}, \mathcal{H}, P, \pi)$.

\begin{thm}\label{l-inf-imp}
    Let $(\mathfrak{A}, \{ \mu_k \}_{k=1}^\infty)$ be an $A_\infty$-algebra. Then there is an $L_\infty$-algebra
    \begin{align*}
        \big( \mathcal{H} = \oplus_{n \in \mathbb{Z}} C^n ( \overline{\mathfrak{A}} ,  \mathfrak{A}   ), \{ l_k \}_{k=1}^\infty  \big)
    \end{align*}
    with the structure maps given by
\begin{align*}
    l_k (f_1, \ldots, f_k) := P [ \cdots [[ \pi, f_1], f_2], \ldots, f_k ],
\end{align*}
for any $k \geq 1$ and homogeneous elements $f_1, \ldots, f_k \in \mathcal{H}$.
\end{thm}

This $L_\infty$-algebra is a generalization of the differential graded Lie algebra obtained in Theorem \ref{mc-gla}. Inspired by this, we may define the following.

\begin{defn}\label{defi-diff-homo}
   Let $(\mathfrak{A}, \{ \mu_k \}_{k=1}^\infty)$ be an $A_\infty$-algebra. Then a {\bf difference operator} on $\mathfrak{A}$ is an element $\mathfrak{D} \in C^0 (\mathfrak{A}, \mathfrak{A})$ such that the corresponding element (also denoted by the same notation) $\mathfrak{D} \in \mathcal{H}_0 = C^0 (\overline{\mathfrak{A}}, \mathfrak{A})$ is a Maurer-Cartan element of the $L_\infty$-algebra $(\mathcal{H} =  \oplus_{n \in \mathbb{Z}} C^n ( \overline{\mathfrak{A}},  \mathfrak{A}   ), \{ l_k \}_{k=1}^\infty)$ considered in Theorem \ref{l-inf-imp}, i.e.,
    \begin{align}\label{defi-homo-diff}
         \sum_{k=1}^\infty \frac{1}{k! } ~ \! l_k (\mathfrak{D}, \ldots, \mathfrak{D}) = 0.
    \end{align}
\end{defn}
\noindent We must note that any element $\mathfrak{D} \in C^0 (\mathfrak{A}, \mathfrak{A})$ is of the form $\mathfrak{D} = \sum_{p=1}^\infty \mathfrak{D}_p$, where $\mathfrak{D}_p : \mathfrak{A}^{\otimes p} \rightarrow \mathfrak{A}$ is a degree $0$ graded linear map. Hence the corresponding element $\mathfrak{D} \in \mathcal{H}_0 =  C^0 (\overline{\mathfrak{A}}, \mathfrak{A})$ is also given by $\mathfrak{D} = \sum_{p=1}^\infty \mathfrak{D}_p$, where $\mathfrak{D}_p :  \overline{\mathfrak{A}}^{\otimes p} \rightarrow \mathfrak{A}$ is the degree $0$ graded linear map given by $\mathfrak{D}_p ( \overline{a_1}, \ldots, \overline{a_p}) := \mathfrak{D}_p (a_1, \ldots, a_p)$, for $\overline{a_1}, \ldots, \overline{a_p} \in \overline{\mathfrak{A}}$.

\medskip

A pair $((\mathfrak{A}, \{ \mu_k \}_{k=1}^\infty), \mathfrak{D})$ consisting of an $A_\infty$-algebra with a difference operator is said to be a {\bf difference $A_\infty$-algebra}. It follows that a difference $A_\infty$-algebra can be regarded as the homotopification (in full generality) of a difference algebra. As mentioned earlier, the convention used in our definition of an $A_\infty$-algebra is related to the standard definition by a degree shift. Hence, one may settle the definition of difference operators on a standard $A_\infty$-algebra. In the standard definition, if the underlying graded vector space is concentrated in arities $0$ and $1$, we exactly obtain $2$-term $A_\infty$-algebras considered in Definition \ref{defi-2term-ainf}. In the same spirit, by expanding the condition (\ref{defi-homo-diff}), one may explicitly write the description of difference operators on a $2$-term $A_\infty$-algebra. More precisely, we have the following.

\begin{defn}
    Let $\mathfrak{a} = (A_1 \xrightarrow{\delta} A_0, \odot, \mu)$ be a $2$-term $A_\infty$-algebra. Then a {\bf difference operator} on $\mathfrak{a}$ is a triple $\mathfrak{d} = (d_0, d_1, d_2)$, where $d_0 : A_0 \rightarrow A_0$ and $d_1 : A_1 \rightarrow A_1$ are linear maps satisfying $d_0 \circ \delta = \delta \circ d_1$, and $d_2 : A_0 \times A_0 \rightarrow A_1$ is a bilinear map such that for all $x, y, z \in A_0$ and $h \in A_1$,
    \begin{align}
        d_0 (x) \odot y + x \odot d_0 (y) + d_0 (x) \odot d_0 (y) - d_0 (x \odot y) =~& \delta (d_2 (x, y)), \label{d1} \tag{D1} \\
        d_0 (x) \odot h + x \odot d_1 (h) + d_0 (x) \odot d_1 (h) - d_1 (x \odot h) =~& d_2 (x, \delta (h)), \label{d2} \tag{D2}\\
        d_1 (h) \odot x + h \odot d_0 (x) + d_1 (h) \odot d_0 (x) - d_1 (h \odot x) =~& d_2 (\delta (h), x), \label{d3} \tag{D3}
        \end{align}
        \begin{align}
        &(x + d_0 (x)) \odot d_2 (y, z) - d_2 (x \odot y, z) + d_2 (x, y \odot z) - d_2 (x, y) \odot (z + d_0 (z)) \nonumber \\
        &= \mu (d_0 (x), y, z) + \mu (x, d_0 (y), z) + \mu (x,y, d_0 (z)) + \mu (d_0 (x), d_0 (y), z) + \mu ( d_0 (x), y, d_0 (z)) \nonumber \\ 
        & \qquad + \mu (x, d_0 (y), d_0 (z)) 
        + \mu (d_0 (x), d_0 (y), d_0 (z)) - d_1 (\mu (x, y, z) ). \label{d4} \tag{D4}
    \end{align}
\end{defn}
\noindent A pair $(\mathfrak{a}, \mathfrak{d})$ consisting of a $2$-term $A_\infty$-algebra $\mathfrak{a}$ with a differential operator $\mathfrak{d}$ is said to be a {\bf $2$-term difference $A_\infty$-algebra}.

\begin{defn}
    Let $(\mathfrak{a}, \mathfrak{d}) = (  (A_1 \xrightarrow{\delta} A_0, \odot, \mu),  (d_0, d_1, d_2) )$ and  $(\mathfrak{a}', \mathfrak{d}') = (  (A'_1 \xrightarrow{\delta'} A'_0, \odot', \mu'),  (d'_0, d'_1, d'_2) )$ be $2$-term difference $A_\infty$-algebras. Then a {\bf difference $A_\infty$-homomorphism} from the former one to the later one is a quadruple $\varphi = (\varphi_0, \varphi_1, \varphi_2, \varphi_3)$, where $(\varphi_0, \varphi_1, \varphi_2)$ is an $A_\infty$-homomorphism of the underlying $A_\infty$-algebras from $\mathfrak{a} = (A_1 \xrightarrow{\delta} A_0, \odot, \mu)$ to $\mathfrak{a}' = (A_1' \xrightarrow{\delta'} A'_0, \odot', \mu')$, and $\varphi_3: A_0 \rightarrow A'_1$ is a linear map satisfying
    \begin{align}
        \varphi_0 (d_0 (x)) - d_0' (\varphi_0 (x)) =~& \delta' (\varphi_3 (x)), \label{hd-eq1}\\
        \varphi_1 (d_1 (h)) - d_1' (\varphi_1 (h)) =~& \varphi_3 (\delta (h)), \label{hd-eq2}\\
        \varphi_1 (d_2 (x, y)) - d_2' ( \varphi_0 (x), \varphi_0 (y)) =~& \big\{ \varphi_2 (d_0 (x), y) + \varphi_2 (x, d_0 (y)) + \varphi_2 (d_0 (x), d_0 (y)) - \varphi_3 (x \odot y)   \big\} \nonumber \\
        & + \big\{ \varphi_3 (x) \odot'\varphi_0 (y) + \varphi_0 (x) \odot' \varphi_3 (y) - d_1' (\varphi_2 (x, y))     \big\}, \label{hd-eq3}
    \end{align}
    for all $x, y \in A_0$ and $h \in A_1.$
\end{defn}

With the above definitions, we have the following.

\begin{thm}\label{thm-categ-2term}
    The collection of all $2$-term difference $A_\infty$-algebras and difference $A_\infty$-homomorphisms between them is a category. (We denote this category by {\sf 2TermDiffA$_\infty$}).
\end{thm}

\begin{proof}
    As before, we only define the composition of difference $A_\infty$-homomorphisms and describe identity homomorphisms. Let $ \varphi = (\varphi_0, \varphi_1, \varphi_2, \varphi_3) : (\mathfrak{a}, \mathfrak{d}) \rightarrow (\mathfrak{a}', \mathfrak{d}')$ and $\psi = (\psi_0, \psi_1, \psi_2, \psi_3) : (\mathfrak{a}', \mathfrak{d}') \rightarrow (\mathfrak{a}'' , \mathfrak{d}'' )$ be two difference $A_\infty$-homomorphisms. We define their composition $\psi \circ \varphi :  (\mathfrak{a}, \mathfrak{d}) \rightarrow  (\mathfrak{a}'', \mathfrak{d}'')$ by taking
    \begin{center}
        $(\psi \circ \varphi)_0 := \psi_0 \circ \varphi_0, \quad (\psi \circ \varphi)_1 := \psi_1 \circ \varphi_1, \quad
        (\psi \circ \varphi)_2 (x, y) := \psi_2 ( \varphi_0 (x), \varphi_0 (y)) + \psi_1 (\varphi_2 (x, y)), $
        \end{center}
        \begin{center}
        $(\psi \circ \varphi)_3 (x) := \psi_3 (\varphi_0 (x)) + \psi_1 (\varphi_3 (x)),$
    \end{center}
    for $x, y \in A_0$. Next, for any $2$-term difference $A_\infty$-algebra $(\mathfrak{a}, \mathfrak{d}) =  (  (A_1 \xrightarrow{\delta} A_0, \odot, \mu),  (d_0, d_1, d_2) )$, we define the identity homomorphism $1_{ (\mathfrak{a}, \mathfrak{d})} : (\mathfrak{a}, \mathfrak{d}) \rightarrow (\mathfrak{a}, \mathfrak{d})$ by taking
    \begin{align*}
        (1_{ (\mathfrak{a}, \mathfrak{d})})_0 = \mathrm{Id}_{A_0}, \quad (1_{ (\mathfrak{a}, \mathfrak{d})})_1 =\mathrm{Id}_{A_1}, \quad (1_{ (\mathfrak{a}, \mathfrak{d})})_2 = 0 \quad \text{ and } \quad (1_{ (\mathfrak{a}, \mathfrak{d})})_3 = 0.
    \end{align*}
    With this setup, one may easily verify that {\sf 2TermDiffA$_\infty$} is a category.
\end{proof}

We are now ready to consider the main result of this paper.

\begin{thm}\label{thm-main}
    The categories {\sf DiffAss2} and {\sf 2TermDiffA$_\infty$} are equivalent.
\end{thm}

\begin{proof}
Let $((C, \bullet, \mathcal{A}), D, \mathcal{D})$ be a difference associative $2$-algebra. Since $(C = (C_1 \rightrightarrows C_0), \bullet, \mathcal{A})$ is an associative $2$-algebra, it follows from \cite{das-hom,kh} that $(A_1 \xrightarrow{\delta} A_0, \odot, \mu)$ is a $2$-term $A_\infty$-algebra, where
\begin{align}
   & A_0 = C_0, \quad A_1 = \mathrm{ker} (s) \subset C_1, \quad \delta (h) := t |_{\mathrm{ker}(s)} (h), \quad x \odot y := x \bullet y, \nonumber\\
   & \quad  x \odot h := i_x \bullet h, \quad h \odot x := h \bullet i_x ~~~ \text{ and } ~~~ \mu(x, y, z) := \overrightarrow{\mathcal{A}_{x, y, z}}, \label{some1}
\end{align}
for $x, y, z \in A_0$ and $h \in A_1$. Next, for the given difference operator $(D, \mathcal{D})$ on the associative $2$-algebra, we define two linear maps $d_0 : A_0 \rightarrow A_0$ and $d_1 : A_1 \rightarrow A_1$, and a bilinear map $d_2 : A_0 \times A_0 \rightarrow A_1$ by
\begin{align}
    d_0 (x) := D_0 (x), \quad d_1 (h) := D_1 |_{\mathrm{ker}(s)}  (h), \quad d_2 (x, y) := \overrightarrow{ \mathcal{D}_{x, y}}, \label{some2}
\end{align}
for $x, y, z \in A_0$ and $h \in A_1$. For any $x, y \in A_0$, we observe that 
\begin{align*}
    d_0 (x) \odot y + x \odot d_0 (y) + d_0 (x) \odot d_0 (y) - d_0 (x \odot y ) =~& D_0 (x) \bullet y + x \bullet D_0 (y) + D_0 (x) \bullet D_0 (y) - D_0 (x \bullet y) \\
    \stackrel{(\ref{defi-D})}{=}& (t-s) ~ \! \mathcal{D}_{x, y} = t ~ \! \overrightarrow{\mathcal{D}_{x, y}} = \delta (d_2 (x, y)).
\end{align*}
Hence, the identity (\ref{d1}) holds. Next, let $x \in A_0$ and $h \in A_1 = \mathrm{ker}(s)$. Then $h$ must be of the form $h = \overrightarrow{f}$, for some morphism $f : y \rightarrow z$. From the naturality of $\mathcal{D}$, we have
\begin{align*}
    \big( i_{D_0 (x)} \bullet f + i_x \bullet D_1 (f) + i_{D_0 (x)} \bullet D_1 (f)     \big) \circ \mathcal{D}_{x, y} = \mathcal{D}_{x, z} \circ \big( D_1 (i_x \bullet f)    \big).
\end{align*}
By taking the arrow parts of both sides of the above equation, we get
\begin{align*}
    \overrightarrow{ i_{D_0 (x)} \bullet f } + \overrightarrow{ i_x \bullet D_1 (f)} + \overrightarrow{ i_{D_0 (x)} \bullet D_1 (f)  } + \overrightarrow{ \mathcal{D}_{x, y}} = \overrightarrow{  \mathcal{D}_{x, z}} + \overrightarrow{ D_1 (i_x \bullet f)   }.
\end{align*}
i.e.,
\begin{align*}
    i_{D_0 (x)} \bullet \overrightarrow{f} + i_x \bullet D_1 ( \overrightarrow{f} ) + i_{D_0 (x)} \bullet D_1 ( \overrightarrow{f} ) - D_1 (i_x \bullet \overrightarrow{f}) = \overrightarrow{ \mathcal{D}_{x, z-y}}.
\end{align*}
This implies that
\begin{align*}
    d_0 (x) \odot h + x \odot d_1 (h) + d_0 (x) \odot d_1 (h) - d_1 (x \odot h) = d_2 (x, \delta (h)),
\end{align*}
which verifies the identity (\ref{d2}). Similarly, one may also verify the identity (\ref{d3}). Finally, the diagram given in (\ref{2-diff}) is equivalent to
\begin{align*}
    &i_{x+ D_0 (x)} \bullet \overrightarrow{ \mathcal{D}_{y, z}} ~+~ \overrightarrow{ \mathcal{D}_{x, y \bullet z}} ~+~ D_1 (\overrightarrow{ \mathcal{A}_{x, y, z} }) = \overrightarrow{ \mathcal{A}_{D_0 (x), y, z} } ~+~ \overrightarrow{ \mathcal{A}_{x, D_0 (y), z} } ~+~ \overrightarrow{ \mathcal{A}_{x, y, D_0 (z) }} ~+~ \overrightarrow{ \mathcal{A}_{D_0 (x), D_0 (y), z} } \\
    & \qquad \quad +~ \overrightarrow{ \mathcal{A}_{D_0 (x), y, D_0 (z)}} ~+~ \overrightarrow{ \mathcal{A}_{x, D_0 (y), D_0 (z)}} ~+~ \overrightarrow{ \mathcal{A}_{D_0 (x), D_0 (y), D_0 (z)}} ~+~ \overrightarrow{ \mathcal{D}_{x \bullet y , z}} ~+~ \overrightarrow{ \mathcal{D}_{x, y}} \bullet i_{z + D_0 (z)}.
\end{align*}
Hence, we obtain that
\begin{align*}
     &(x + d_0 (x)) \odot d_2 (y, z)  + d_2 (x, y \odot z) + d_1 (\mu (x, y, z) ) \nonumber \\
        &= \mu (d_0 (x), y, z) + \mu (x, d_0 (y), z) + \mu (x,y, d_0 (z)) + \mu (d_0 (x), d_0 (y), z) + \mu ( d_0 (x), y, d_0 (z)) \nonumber \\ 
        & \qquad + \mu (x, d_0 (y), d_0 (z)) 
        + \mu (d_0 (x), d_0 (y), d_0 (z)) + d_2 (x \odot y, z) + d_2 (x, y) \odot (z + d_0 (z)) .
\end{align*}
Therefore, the identity (\ref{d4}) also holds. Hence $((A_1 \xrightarrow{\delta} A_0, \odot, \mu ), (d_0, d_1, d_2) )$ is a $2$-term difference $A_\infty$-algebra.

\medskip

Next, let $((C, \bullet, \mathcal{A}), D, \mathcal{D})$ and $( (C', \bullet', \mathcal{A}') , D', \mathcal{D}')$ be two difference associative $2$-algebras, and 
\begin{align*}
    F = (F_0, F_1, F_2, F_3) : ((C, \bullet, \mathcal{A}), D, \mathcal{D}) \rightarrow ( (C', \bullet', \mathcal{A}'), D', \mathcal{D}') 
\end{align*}
be a homomorphism of difference associative $2$-algebras. Since $(F_0, F_1, F_2) : (C, \bullet, \mathcal{A}) \rightarrow (C', \bullet' , \mathcal{A}')$ is a homomorphism between the underlying associative $2$-algebras, it follows from \cite{das-hom,kh} that the triple
\begin{align*}
     (\varphi_0, \varphi_1, \varphi_2) : (A_1 \xrightarrow{\delta} A_0, \odot, \mu ) \rightarrow (A'_1 \xrightarrow{\delta'} A_0', \odot', \mu' )
\end{align*}
is an $A_\infty$-homomorphism between the corresponding $2$-term $A_\infty$-algebras, where
\begin{align*}
    \varphi_0 (x) := F_0 (x), \quad \varphi_1 (h) := F_1 \big|_{\mathrm{ker} (s)} (h) ~~~~ \text{ and } ~~~~\varphi_2 (x, y) := \overrightarrow{F_2 (x, y)}, 
\end{align*}
for $x, y \in A_0$ and $h \in A_1$. Next, we define a map $\varphi_3 : A_0 \rightarrow A_1'$ by 
\begin{align*}
    \varphi_3 (x) := \overrightarrow{ F_3 (x) }, \text{ for } x \in A_0.
\end{align*}
Then for any $x \in A_0$, we see that
\begin{align*}
    \varphi_0 (d_0 (x)) - d_0'(\varphi_0 (x)) = F_0 (D_0 (x)) - D_0' (F_0 (x)) \stackrel{(\ref{f3})}{ = } (t' - s') ~\! F_3 (x) = t' ~\! \overrightarrow{ F_3 (x)} = \delta' (\varphi_3 (x)).
\end{align*}
Next, take $h \in A_1 = \mathrm{ker}(s)$. As before, let $h = \overrightarrow{f}$, for some morphism $f : x \rightarrow y$. Note that the naturality of $F_3$ implies that 
\begin{align*}
    F_1 (D_1 (f)) ~ \! F_3 (x) = F_3 (y) ~ \! D_1' (F_1 (f)).
\end{align*}
By considering the arrow parts in both sides, we obtain $\overrightarrow{ F_3 (x)} + F_1 ( D_1 (\overrightarrow{f})) = D_1'(F_1 (\overrightarrow{f}) ) + \overrightarrow{F_3 (y)}$. This shows that
\begin{align*}
    F_1 ( D_1 (\overrightarrow{f})) - D_1'(F_1 (\overrightarrow{f}) )  = \overrightarrow{F_3 (y-x)}, \quad \text{i.e.,} \quad \varphi_1 ( d_1 (h)) - d_1'(\varphi_1 (h)) = \varphi_3 (\delta ( h )).
\end{align*}
Hence, both the identities (\ref{hd-eq1}) and (\ref{hd-eq2}) hold. Finally, the diagram given in (\ref{2-diff-homo}) implies that
\begin{align*}
    &\big(  \overrightarrow{ F_2 (D_0 (x), y)} + \overrightarrow{ F_2 (x, D_0 (y))} + \overrightarrow{ F_2 (D_0 (x), D_0 (y)) } ~ \big) + \big( \overrightarrow{F_3 (x) } \bullet' i_{F_0 (y)} + i_{F_0 (x)} \bullet' \overrightarrow{ F_3 (y)} ~ \big) + \overrightarrow{ \mathcal{D}'_{F_0 (x), F_0 (y)}} \\
    & \qquad \qquad \qquad \qquad = F_1 ( \overrightarrow{ \mathcal{D}_{x, y}}) + \overrightarrow{ F_3 (x \bullet y)} + D_1' ( \overrightarrow{ F_2 (x, y)}).
\end{align*}
which in turn shows the identity (\ref{hd-eq3}). This concludes that 
\begin{align*}
    (\varphi_0, \varphi_1, \varphi_2, \varphi_3) : ((A_1 \xrightarrow{\delta} A_0, \odot, \mu ), (d_0, d_1, d_2) ) \rightarrow ((A'_1 \xrightarrow{\delta'} A'_0, \odot', \mu' ), (d'_0, d'_1, d'_2) )
\end{align*}
is a difference $A_\infty$-homomorphism. It is easy to see that the above construction preserves the identity homomorphisms and composition of homomorphisms. Thus, we get a functor {\sf S} : {\sf DiffAss2} $\rightarrow$ {\sf 2TermDiffA$_\infty$}.

\medskip

On the other hand, suppose we start with a $2$-term difference $A_\infty$-algebra $ ((A_1 \xrightarrow{\delta} A_0, \odot, \mu ), (d_0, d_1, d_2) )$. As $(A_1 \xrightarrow{\delta} A_0, \odot, \mu )$ is a $2$-term $A_\infty$-algebra, we get that $(C = (C_1 \rightrightarrows C_0), \bullet, \mathcal{A})$ is an associative $2$-algebra, where $C_0 = A_0$ and $C_1 = A_0 \oplus A_1$. The source, target, identity-assigning map and the composition map are respectively given by
\begin{align*}
   & s (x, h) := x, \quad t (x, h) := x + \delta (h), \quad i_x := (x, 0) ~~~ \text{ and }\\
   & \quad (x, h) \circ (y, k) = (x, h+k), \text{ when } t (x, h) = s (y,k).
\end{align*}
The bilinear functor $\bullet : C \times C \rightarrow C$ is given by
\begin{align*}
     x \bullet y := x \odot y ~~~ \text{ and } ~~~ (x, h) \bullet (y, k) := (x \odot y, x \odot k + h \odot y + \delta (h) \odot k),
\end{align*}
for objects $x, y \in C_0$ and morphisms $(x, h), (y, k) \in C_1$. Moreover, the associator for $C$ is given by
\begin{align*}
    \mathcal{A}_{x, y, z} := \big(  (x \bullet y) \bullet z ~\!, ~ \! \mu (x, y, z)   \big).
\end{align*}
Next, from the difference operator $(d_0, d_1, d_2)$ on the $2$-term $A_\infty$-algebra $(A_1 \xrightarrow{\delta} A_0, \odot, \mu )$, we define a linear functor $D = (D_0, D_1) : C \rightarrow C$ by
\begin{align*}
    D_0 (x) := d_0 (x) ~~~ \text{ and } ~~~D_1 (x, h) := (d_0 (x), d_1 (h)),
\end{align*}
for $x \in C_0$ and $(x, h) \in C_1$. Moreover, we set a natural isomorphism
\begin{align*}
    \mathcal{D}_{x, y} : D_0 (x \bullet y) \rightarrow D_0 (x) \bullet y + x \bullet D_0 (y) + D_0 (x) \bullet D_0 (y) ~~~~   \text{ by } ~~~~  \mathcal{D}_{x, y} := \big(  D_0 (x \bullet y) ~\!, ~\! d_2 (x, y)  \big).
\end{align*}
Then it turns out that the pair $(D, \mathcal{D})$ is a difference operator on the associative $2$-algebra $(C, \bullet, \mathcal{A})$ constructed above. In other words, $((C , \bullet, \mathcal{A}), D, \mathcal{D})$ is a difference associative $2$-algebra.

\medskip

Let $ (\varphi_0, \varphi_1, \varphi_2, \varphi_3) : ((A_1 \xrightarrow{\delta} A_0, \odot, \mu ), (d_0, d_1, d_2) ) \rightarrow ((A'_1 \xrightarrow{\delta'} A'_0, \odot', \mu' ), (d'_0, d'_1, d'_2) )$ be a difference $A_\infty$-homomorphism. Since $(\varphi_0, \varphi_1, \varphi_2) : (A_1 \xrightarrow{\delta} A_0, \odot, \mu ) \rightarrow (A'_1 \xrightarrow{\delta'} A'_0, \odot', \mu' )$ is an $A_\infty$-homomorphism, it follows that the triple $(F_0, F_1, F_2) $ is a homomorphism between the corresponding associative $2$-algebras, where
\begin{align*}
    &F_0 : A_0 \rightarrow A_0' ~~ \text{ is given by } ~~ F_0 (x) := \varphi_0 (x),\\
    &F_1 : A_0 \oplus A_1 \rightarrow A_0' \oplus A_1' ~~ \text{ is given by } ~~  F_1 (x, h) := ( \varphi_0 (x), \varphi_1 (h)), \\
    &F_2 (x, y) : F_0 (x) \bullet' F_0 (y) \rightarrow F_0 (x \bullet y)  ~~ \text{ is given by } ~~ F_2 (x, y) := \big(  F_0 (x) \bullet' F_0 (y) ~ \! , ~ \! \varphi_2 (x \bullet y)  \big).
\end{align*}
We also define a natural transformation $F_3 (x) : D_0' (F_0 (x)) \rightarrow F_0 (D_0 (x))$ by $F_3 (x) := \big( D_0' (F_0 (x)) ~ \!, ~ \! \varphi_3 (x)    \big).$ (Note that the naturality of $F_3$ follows from the condition (\ref{hd-eq2})). Then the identity (\ref{hd-eq3}) ensures the commutativity of the diagram given in (\ref{2-diff-homo}). Hence the quadruple
\begin{align*}
     (F_0, F_1, F_2, F_3) : ((C , \bullet, \mathcal{A}), D, \mathcal{D}) \rightarrow ((C' , \bullet', \mathcal{A}'), D', \mathcal{D}')
\end{align*}
is a homomorphism of difference associative $2$-algebras. It is also easy to see that the above construction preserves the identity homomorphisms and composition of homomorphisms. As a result, we get a functor {\sf T} : {\sf 2TermDiffA$_\infty$} $\rightarrow$ {\sf DiffAss2}. 

\medskip

Thus, we are now left to show the existence of two natural isomorphisms $\alpha : {\sf TS} \Rightarrow  1_{\sf DiffAss2}$ and $\beta: {\sf ST} \Rightarrow 1_{\sf 2TermDiffA_\infty }$. Let $(C = (C , \bullet, \mathcal{A}), D, \mathcal{D})$ be a difference associative $2$-algebra. By applying the functor {\sf S}, we suppose that
\begin{align*}
    {\sf S} (C, D, \mathcal{D}) := ((A_1 \xrightarrow{\delta} A_0, \odot, \mu ), (d_0, d_1, d_2) )
\end{align*}
is a $2$-term difference $A_\infty$-algebra, where the structure operations are given in (\ref{some1})-(\ref{some2}). By applying the functor ${\sf T}$ to this structure, we obtain a new difference associative $2$-algebra, say $(C' =  (C' , \bullet', \mathcal{A}'), D', \mathcal{D}')$. Then we have
\begin{align*}
    &C_0'= A_0 = C_0, \quad C_1'= A_0 \oplus A_1 = C_0 \oplus \mathrm{ker}(s),\\
    &s' (x, h) = x, \quad t'(x, h) = x + \delta (h) = x + t (h), \quad i_x = (x, 0),\\
    &x \bullet' y = x \odot y =  x \bullet y, \quad (x, h) \bullet' (y, k) = ( x \bullet y ~ \! , ~ \! i_x k + h \bullet i_y + i_{\delta (h)} \bullet k), \\
    &D_0'(x) = D_0 (x), \quad D_1'(x, h) = (D_0 (x) , D_1 (h)) ~~~ \text{ and } ~~~ \mathcal{D}'_{x, y} = \big( D_0 (x) \bullet D_0 (y) ~ \!, ~\! d_2 (x, y)    \big) = \mathcal{D}_{x, y},
\end{align*}
for $x, y \in C_0' = C_0$ and $(x, h), (y, k) \in C_1' = C_0 \oplus \mathrm{ker} (s)$. It is easy to see that 
\begin{align*}
    \alpha_C = \big(  ( \alpha_C)_0, ( \alpha_C)_1, ( \alpha_C)_2, ( \alpha_C)_3   \big) : (C', D', \mathcal{D}') \rightarrow (C, D, \mathcal{D})
\end{align*}
is an isomorphism of difference associative $2$-algebras, where $ ( \alpha_C)_0 (x) := x$, $ ( \alpha_C)_1 (x, h) := i_x + h$ and $ ( \alpha_C)_2$, $ ( \alpha_C)_3$ are defined to be the identity maps. This construction leads to a natural isomorphism $\alpha : {\sf TS} \Rightarrow  1_{\sf DiffAss2}$.

\medskip

On the other hand, let $(\mathfrak{a}, \mathfrak{d}) = ((A_1 \xrightarrow{\delta} A_0, \odot, \mu ), (d_0, d_1, d_2) )$ be a $2$-term difference $A_\infty$-algebra. After applying the functor $\mathsf{T}$, we assume that the difference associative $2$-algebra $\mathsf{T} (\mathfrak{a}, \mathfrak{d})$ is given by
\begin{align*}
    \mathsf{T} (\mathfrak{a}, \mathfrak{d}) = (C = (C, \bullet, \mathcal{A}), D, \mathcal{D}).
\end{align*}
By further applying the functor $\mathsf{S}$ to this structure, we obtain a $2$-term difference $A_\infty$-algebra, say $(\mathfrak{a}', \mathfrak{d}')$. Then it is easy to see that $(\mathfrak{a}', \mathfrak{d}')$ is exactly same with $(\mathfrak{a}, \mathfrak{d})$. Thus, we have the identity isomorphism
\begin{align*}
    \beta_\mathfrak{a} = \big( (\beta_\mathfrak{a})_0, (\beta_\mathfrak{a})_1, (\beta_\mathfrak{a})_2, (\beta_\mathfrak{a})_3 \big) : (\mathfrak{a}', \mathfrak{d}') \rightarrow (\mathfrak{a}, \mathfrak{d})
\end{align*}
of $2$-term difference $A_\infty$-algebras, where $(\beta_\mathfrak{a})_0 (x) = x$, $(\beta_\mathfrak{a})_1 (h) = h$ and $(\beta_\mathfrak{a})_2$, $(\beta_\mathfrak{a})_3$ are trivial maps. This yields a natural isomorphism $\beta: {\sf ST} \Rightarrow 1_{\sf 2TermDiffA_\infty }$. Hence, the proof is done.
\end{proof}

\medskip

\begin{defn}\label{defi-sk-st}
Let $(\mathfrak{a}, \mathfrak{d}) = (( A_1 \xrightarrow{ \delta} A_0, \odot, \mu ) , (d_0, d_1, d_2) )$ be a $2$-term difference $A_\infty$-algebra. Then it is said to be 

- {\bf skeletal} if $\delta = 0$,

- {\bf strict} if $\mu = 0$ and $d_2 = 0$.
\end{defn}
\noindent The collection of all strict $2$-term difference $A_\infty$-algebras and difference $A_\infty$-homomorphisms between them is also a category, denoted by ${\sf S2TermDiffA_\infty }$. This turns out to be a subcategory of  ${\sf 2TermDiffA_\infty }$. Inspired by Theorem \ref{thm-main}, one can then prove the following result.

\begin{corollary}\label{cor-main}
The categories {\sf SDiffAss2} and ${\sf S2TermDiffA_\infty }$ are equivalent.
\end{corollary}

All the categories and the equivalences considered above can be summarized as
\[
\xymatrix{
{\sf DiffAss2} \ar[rrr]_\cong^{\mathrm{Theorem }~  \ref{thm-main}} & & & {\sf 2TermDiffA_\infty } \\
{\sf SDiffAss2} \ar@{^{(}->}[u] \ar[rrr]^\cong_{\mathrm{Corollary }~  \ref{cor-main} } & & & {\sf S2TermDiffA_\infty.} \ar@{^{(}->}[u]
}
\]

\medskip

\section{Characterizations of some difference $A_\infty$-algebras and $2$-term bimodule up to homotopy}\label{sec5}

In this section, we first characterize skeletal and strict $2$-term difference $A_\infty$-algebras by respectively third cocycles and crossed modules of difference algebras. Next, we define the notion of a $2$-term bimodule up to homotopy over an associative algebra, that provides a $2$-term $A_\infty$-algebra as the corresponding semidirect product. Subsequently, we extend this construction to the context of difference algebras.

\begin{thm}\label{thm-skeletal}
    There is a 1-1 correspondence between skeletal $2$-term difference $A_\infty$-algebras and triplets $((A, d), (M, \Delta), (\mu, \chi))$ consisting of a difference algebra $(A, d)$, a bimodule $(M, \Delta )$ over it and a $3$-cocycle $(\mu, \chi) \in Z^3((A, d); (M, \Delta) )$.
\end{thm}

\begin{proof}
    Let $(( A_1 \xrightarrow{\delta = 0} A_0, \odot, \mu), (d_0, d_1, d_2) )$ be a skeletal $2$-term difference $A_\infty$-algebra. Since $\delta = 0$, it follows from the identities (\ref{a4}), (\ref{d1}) that $(A_0 = (A_0, \odot), d_0)$ is a difference algebra. On the other hand, the identities (\ref{a5}), (\ref{a6}), (\ref{a7}), (\ref{d2}) and (\ref{d3}) imply that $(A_1, d_1)$ is a bimodule over the difference algebra $(A_0, d_0)$ obtained above. Finally, the identity (\ref{a8}) simply means that $\delta_\mathrm{Hoch} (\mu ) (x, y, z, t ) = 0$, while the identity (\ref{d4}) can be equivalently understood as $\big( \delta_\mathrm{Hoch}^{d_0} (d_2) - \partial^{d_0, d_1} (\mu) \big) (x, y, z) = 0$. Here $\delta_\mathrm{Hoch}$ is the Hochschild coboundary operator of the associative algebra $A_0$ with coefficients in the $A_0$-bimodule $A_1$. Hence, we have
    \begin{align*}
        \delta_\mathrm{Diff} (\mu, d_2) = \big(  \delta_\mathrm{Hoch} (\mu) ~ \!, ~ \! \delta_\mathrm{Hoch}^{d_0} (d_2) - \partial^{d_0, d_1} (\mu)  \big) = 0
    \end{align*}
    which shows that $(\mu, d_2) \in Z^3_{\mathrm{Diff}} ((A_0, d_0); (A_1, d_1) )$ is a $3$-cocycle. Thus, we obtain the required triple $((A_0, d_0), (A_1, d_1), (\mu, d_2) ).$

    \medskip

    Conversely, suppose there is a triple $((A, d), (M, \Delta), (\mu, \chi) )$ of a difference algebra, a bimodule over it and a $3$-cocycle. Then it is easy to see that
    \begin{align*}
         ((M \xrightarrow{\delta = 0} A, \odot, \mu ), (d_0 = d, d_1 = \Delta, d_2= \chi) )
    \end{align*}
    is a skeletal $2$-term difference $A_\infty$-algebra, where the bilinear map $\odot$ is given by
    \begin{align*}
        a \odot b = ab, \quad a \odot u = au ~~~ \text{ and } ~~~ u \odot a = ua, \text{ for } a, b \in A, ~ \! u \in M.
    \end{align*}
    The above two constructions are inverses of each other. This completes the proof.
\end{proof}

Next, we shall consider crossed modules of difference algebras and their correspondence with strict $2$-term difference $A_\infty$-algebras.

\begin{defn}
    A {\bf crossed module of difference algebras} is a triple $((A, d), (A', d'), \partial )$, where
    \begin{itemize}
        \item[-] $(A, d)$ and $(A', d')$ are both difference algebras, and $(A', d')$ is also a bimodule over the difference algebra $(A, d)$,   
        \item[-] $\partial : A' \rightarrow A$ is a homomorphism of difference algebras satisfying additionally
    \end{itemize}
    \begin{align}
        &a (hk) = (ah) k, \qquad h (ak) = (ha)k, \qquad h (ka) = (hk) a, \label{crm1}\\
        & \partial (ah) = a \partial (h), \quad \partial (ha ) = \partial (h) a, \quad  \partial (h) k = h \partial (k) = hk, \label{crm2}
    \end{align}
    for all $a \in A$ and $h, k \in A'$. Here, for simplicity, we denote the associative multiplications of $A$ and $A'$, and also the left and right $A$-actions on $A'$ by the juxtaposition.
\end{defn}

\begin{exam}
(i) Let $(A, d)$ be a difference algebra. Then $((A, d), (A, d), \partial = \mathrm{Id}_A )$ is a crossed module of difference algebras.

\medskip

(ii) Let $(A, A', \partial)$ be a crossed module of associative algebras \cite{das-hom}. Then it is easy to see that the triple $((A, d = - \mathrm{Id}_A), (A' , d' = - \mathrm{Id}_{A'}), \partial )$ is a crossed module of difference algebras.
\end{exam}

\begin{thm}\label{thm-strict}
    There is a 1-1 correspondence between strict $2$-term difference $A_\infty$-algebras and crossed modules of difference algebras.
\end{thm}

\begin{proof}
    Let $(( A_1 \xrightarrow{\delta} A_0, \odot, \mu= 0) , (d_0, d_1, d_2 = 0) )$ be a strict $2$-term difference $A_\infty$-algebra. Since $\mu = 0$ and $d_2 = 0$, it follows from the identities (\ref{a4}) and (\ref{d1}) that $(A_0 = (A_0, \odot), d_0 )$ is a difference algebra. We now define a bilinear map $\odot_1 : A_1 \times A_1 \rightarrow A_1$ by 
    \begin{align*}
        h \odot_1 k := \delta (h) \odot k \stackrel{\text{(\ref{a3})} }{=} h \odot \delta (k), \text{ for } h , k \in A_1.
    \end{align*}
    This multiplication makes the vector space $A_1$ into an associative algebra (that can be seen from either of the identities (\ref{a5}), (\ref{a6}) or (\ref{a7})). Moreover, either of (\ref{d2}) or (\ref{d3}) implies that the linear map $d_1: A_1 \rightarrow A_1$ is a difference operator on $A_1$, and hence $(A_1, d_1)$ is a difference algebra. Further, from (\ref{a1}) or (\ref{a2}), and the condition $d_0 \circ \delta = \delta \circ d_1$, we get that $\delta: A_1 \rightarrow A_0$ is a homomorphism of difference algebras. Next, we consider the bilinear maps
    \begin{align*}
        A_0 \times A_1 \rightarrow A_1, ~ (x, h) \mapsto x \odot h \quad \text{ and } \quad A_1 \times A_0 \rightarrow A_1, ~ (h , x) \mapsto h \odot x.
    \end{align*}
    Then it follows from (\ref{a5})-(\ref{a7}) and (\ref{d2})-(\ref{d3}) that the above bilinear maps make the vector space $A_1$ into an $A_0$-bimodule, and $(A_1, d_1)$ into a bimodule over the difference algebra $(A_0, d_0)$. Moreover, it is easy to see that the identities in (\ref{crm1}) and (\ref{crm2}) hold. Hence $((A_0, d_0 ), (A_1, d_1), \delta )$ is a crossed module of difference algebras.

    \medskip

    Conversely, let $((A, d), (A', d'), \partial)$ be a crossed module of difference algebras. Then one can easily verify that
    \begin{align*}
        ((A' \xrightarrow{\partial} A, \odot, \mu= 0 ) , (d_0 = d, d_1 = d', d_2 = 0) )
    \end{align*}
    is a strict $2$-term difference $A_\infty$-algebra, where $\odot$ is defined by
    \begin{align*}
        a \odot b = a b, \quad a \odot h = a h ~~~ \text{ and } ~~~ h \odot a = h a, \text{ for } a, b \in A, ~ \! h \in A'.
    \end{align*}
    Finally, the above two correspondences are inverses of each other. This proves the result.
\end{proof}

We will now consider a $2$-term bimodule up to homotopy of an associative algebra, and then over a difference algebra. We also discuss the corresponding semidirect products.

\begin{defn}
    Let $A$ be an associative algebra. Then a {\bf $2$-term bimodule up to homotopy} of $A$ consists of a $2$-term complex $M_1 \xrightarrow{\delta} M_0$ endowed with four bilinear maps (all being denoted by juxtaposition)
    \begin{align*}
        A \times M_0 &\rightarrow M_0, \qquad M_0 \times A \rightarrow M_0, \qquad A \times M_1 \rightarrow M_1, \qquad M_1 \times A \rightarrow M_1 \\
        (a, v) & \mapsto av \quad \qquad \quad (v, a) \mapsto va, ~~ \quad \quad \quad (a, \xi) \mapsto a \xi \qquad \qquad (\xi, a) \mapsto \xi a
    \end{align*}
    compatible with $\delta$ in the sense that
    \begin{align}
        \delta (a \xi) = a \delta (\xi) ~~~ \text{ and } ~~~ \delta (\xi a) = \delta (\xi) a, \text{ for } a \in A, ~ \! \xi \in M_1,
    \end{align}
    and three trilinear maps (all being denoted by the same notation $\nu$)
    \begin{align*}
        \nu: A \times A \times M_0 \rightarrow M_1, \qquad \nu : A \times M_0 \times A \rightarrow M_1, \qquad \nu : M_0 \times A \times A \rightarrow M_1 
    \end{align*}
    such that for all $a, b, c \in A$ and $v \in M_0$, ~$\xi \in M_1$,
    \begin{align*}
\begin{cases}
    a (b v) - (ab) v = \delta (\nu (a, b, v)),\\
    a ( v b) - (a v) b = \delta (\nu (a, v, b)),\\
    v (a b) - (v a) b = \delta (\nu (v, a, b)),
\end{cases}
\qquad 
\begin{cases}
    a (b \xi) - (a b) \xi = \nu (a, b, \delta (\xi)), \\
    a (\xi b) - (a \xi) b = \nu (a, \delta (\xi), b), \\
    \xi (a b) - (\xi a) b = \nu (\delta (\xi), a, b),
\end{cases}
    \end{align*}
    and
    \begin{align}
        &a ~ \! \nu (b,c, v) - \nu (ab, c, v) + \nu (a, bc, v) - \nu (a, b, c v) = 0, \label{id-nu1}\\
        &a ~  \! \nu (b, v, c) - \nu (ab, v, c) + \nu (a, bv, c) - \nu (a, b, v c) + \nu (a, b, v) ~ \! c = 0, \label{id-nu2}\\
        & a ~ \! \nu (v, b, c) - \nu (a v, b, c) + \nu (a, v b, c) - \nu (a, v, bc) + \nu (a, v, b) ~ \! c =0, \label{id-nu3}\\
        & ~ ~ \qquad \qquad \quad \nu (v a, b, c) - \nu (v, ab, c) + \nu (v, a, b c) - \nu (v, a, b) ~ \! c = 0. \label{id-nu4}
    \end{align}
    We often denote a $2$-term bimodule up to homotopy as above by $(M_1 \xrightarrow{\delta} M_0, \nu)$.
\end{defn}

\begin{thm}\label{thm-2-hom}
    Let $A$ be an associative algebra and $(M_1 \xrightarrow{\delta} M_0, \nu)$ be a $2$-term bimodule up to homotopy of $A$. Then the triple
    \begin{align*}
       \mathfrak{a} :=  ( M_1 \xrightarrow{\widetilde{\delta } } A \oplus M_0 ~ \! , ~ \! \odot, \mu )
    \end{align*}
    is a $2$-term $A_\infty$-algebra, where $\widetilde{\delta} (\xi) = (0, \delta (\xi) )$, and
    \begin{align*}
       & (a, u) \odot (b, v) := (ab ~ \! , ~ \! a v + u b), \quad (a, u) \odot \xi := a \xi, \quad \xi \odot (a, u) := \xi a, \quad \xi \odot \eta := 0,\\
       & \qquad \qquad \quad  \mu ( (a , u) , (b ,v) , (c , w ) ) := \nu (a, b, w) + \nu (a, v, c) + \nu (u, b, c),
    \end{align*}
    for all $(a, u), (b, v), (c, w) \in A \oplus M_0$ and $\xi, \eta \in M_1$. (This is called the {\bf semidirect product} of $A$ with the given $2$-term bimodule up to homotopy).
\end{thm}

\begin{proof}
    Let $(a, u), (b, v), (c, w) \in A \oplus M_0$ and $\xi, \eta \in M_1$. Then we have
    \begin{align*}
        &\widetilde{\delta} ((a, u) \odot \xi )  = (0, \delta (a \xi)) = (0, a \delta (\xi)) = (a, u) \odot (0, \delta (\xi)) = (a, u) \odot \widetilde{ \delta} (\xi), \\
        &\widetilde{\delta} (\xi \odot (a, u)) = ( 0 , \delta (\xi a)) = (0, \delta (\xi) a) = (0, \delta (\xi)) \odot (a, u) = \widetilde{\delta} (\xi) \odot (a, u), \\
        & \qquad \qquad \widetilde{\delta} (\xi) \odot \eta = (0, \delta (\xi)) \odot \eta = 0 = \xi \odot (0, \delta (\eta)) = \xi \odot \widetilde{\delta} (\eta).
    \end{align*}
    Hence, the identities (\ref{a1}), (\ref{a2}), (\ref{a3}) hold. We also have
    \begin{align*}
        &(a, u) \odot (( b, v) \odot (c, w)) - ((a, u) \odot (b, v)) \odot (c, w) \\
       & = (0 ~ \! , ~ \! a (bw) - (ab) w + a (v c) - (av) c + u (bc) - (ub) c) \\
       & = \big(  0 ~ \! , ~ \! \delta ( \nu (a, b, w)  + \nu (a, v, c) + \nu (u, b, c) )   \big) = \widetilde{\delta} \big(   \mu ((a, u), (b, v), (c, w))  \big),
    \end{align*}
    \begin{align*}
        (a, u) \odot (( b, v) \odot \xi ) - ((a, u) \odot (b, v)) \odot \xi = a (b \xi) - (ab) \xi 
        = \nu (a, b, \delta (\xi)) 
        =~& \mu (( a, u), (b, v), (0, \delta (\xi)) ) \\
        =~& \mu ((a, u), (b, v), \widetilde{\delta} (\xi)).
    \end{align*}
    Hence, the identities (\ref{a4}) and (\ref{a5}) also follow. Similarly, one can verify the identities (\ref{a6}) and (\ref{a7}). Finally, for any $(a, u), (b, v), (c, w), (c', w') \in A \oplus M_0$, we observe that
    \begin{align*}
        &(a, u) \odot \mu (    (b, v) , (c, w), (c', w')) - \mu ( (a, u) \odot (b, v), (c, w), (c', w')) + \mu (   (a, u), (b, v) \odot (c, w) , (c', w')) \\
        & \qquad \qquad - \mu ( (a, u), (b, v) , (c, w) \odot (c', w')) + \mu ( (a, u), (b, v), (c, w)) \odot (c', w') \\
&= \big\{ a ~ \! \nu (b,c, w') - \nu (ab, c, w') + \nu (a, bc, w') - \nu (a, b, cw') \big\} \\
& \quad  + \big\{   a ~ \! \nu (b, w, c') - \nu (ab, w, c') + \nu (a, bw , c') - \nu (a, b, wc') + \nu (a, b, w) ~  \!  c' \big\} \\
& \quad + \big\{  a ~\! \nu (v, c, c')  - \nu (av, c, c')   + \nu (a, v c , c') - \nu (a, v, cc')  + \nu (a, v, c) ~ \! c' \big\} \\
& \quad - \big\{ \nu (ub, c, c' ) - \nu (u, bc, c') + \nu (u, b, cc') - \nu (u, b, c) ~ \! c'  \big\} \\
        &= 0 \quad (\text{by } (\ref{id-nu1})\text{-}(\ref{id-nu4}))
    \end{align*}
    which verifies the identity (\ref{a8}). This completes the proof.
\end{proof}

\begin{defn}
    Let $(A, d)$ be a difference algebra. Then a {\bf $2$-term bimodule up to homotopy} over $(A, d)$ is a tuple $((M_1 \xrightarrow{\delta} M_0, \nu), \Delta_0, \Delta_1, \theta)$, where 

    \medskip

    - $(M_1 \xrightarrow{\delta} M_0, \nu)$ is a $2$-term bimodule up to homotopy of the associative algebra $A$,

    \medskip

    - $\Delta_0 : M_0 \rightarrow M_0$ and $\Delta_1 : M_1 \rightarrow M_1$ are linear maps satisfying $\Delta_0 \circ \delta = \delta \circ \Delta_1$,

    \medskip

    - there are two bilinear maps (both being denoted by the same notation $\theta$)
\begin{align*}
        \theta : A \times M_0 \rightarrow M_1 \qquad \text{ and } \qquad \theta : M_0 \times A \rightarrow M_1
    \end{align*}
    such that for all $a, b \in A$, $v \in M_0$ and $\xi \in M_1$, the following identities hold:
    \begin{align}
        d(a) v +a \Delta_0 (v) + d(a) \Delta_0 (v) - \Delta_0 (av) =~& \delta (\theta (a, v)), \label{condi1}\\
        \Delta_0 (v) a + v d(a) + \Delta_0 (v) d(a) - \Delta_0 (va) =~& \delta (\theta (v, a)),\label{condi2}\\
        d(a) \xi + a \Delta_1 (\xi) + d(a) \Delta_1 (\xi) - \Delta_1 (a \xi) =~& \theta (a, \delta (\xi)), \label{condi3}\\
        \Delta_1 (\xi) a + \xi d(a) + \Delta_1 (\xi) d(a) - \Delta_1 (\xi a) =~& \theta (\delta (\xi), a), \label{condi4}
\end{align}
        \begin{align}
        (a+ d(a)) \theta (b, v) - \theta (ab, v) + \theta (a, bv) =~& \nu (d(a), b, v) + \nu (a, d(b), v) + \nu (a, b, \Delta_0 (v)) \nonumber \\
        + \nu (d(a), d(b), v) + \nu (d (a), b, \Delta_0 (v))  ~&+~ \nu ( a, d(b), \Delta_0 (v)) + \nu (d (a), d (b), \Delta_0 (v)) - d_1 (\nu (a, b, v)),\label{condi5}
\end{align}
\begin{align}
        (a + d(a)) \theta (v, b) - \theta (a v, b) + \theta (a, v b) - \theta (a, v) (b + d(b))=~& \nu (d(a), v, b) + \nu (a, \Delta_0 (v), b) + \nu (a, v, d(b)) \nonumber \\
        + \nu (d (a), \Delta_0 (v), b) + \nu (d (a), v, d (b) ) + \nu (a , \Delta_0 (v), d(b) )  &+ \nu (d (a), \Delta_0 (v), d(b)) -d_1 (\nu (a, v, b)), \label{condi6}
\end{align}
\begin{align}
        - \theta (va, b) + \theta (v, ab) - \theta (v, a) (b+ d(b)) =~& \nu (\Delta_0 (v), a, b) + \nu (v, d(a), b) + \nu (v, a, d(b)) \nonumber \\
        + \nu (\Delta_0 (v), d (a), b) + \nu (\Delta_0 (v), a, d (b) ) &+ \nu (v, d(a), d(b) ) + \nu (\Delta_0 (v), d(a), d(b) ) - d_1 (\nu (v, a, b)). \label{condi7}
    \end{align}
\end{defn}

\medskip

Let $((M_1 \xrightarrow{\delta} M_0, \nu), \Delta_0, \Delta_1, \theta )$ be a $2$-term bimodule up to homotopy over a difference algebra $(A, d)$. Then this $2$-term bimodule up to homotopy is said to be {\em skeletal} if $\delta = 0$. Equivalently, this says that $(M_0, \Delta_0)$ and $(M_1, \Delta_1)$ are both bimodules over the difference algebra $(A, d)$, and $\nu, \theta$ satisfies the identities (\ref{id-nu1})-(\ref{id-nu4}) and (\ref{condi5})-(\ref{condi7}). On the other hand, it is said to be {\em strict} if $\nu = 0$ and $\theta = 0$. This is equivalent to mean that $(M_0, \Delta_0),$ $ ( M_1, \Delta_1)$ are both bimodules over the difference algebra $(A, d)$, and $\delta: M_1 \rightarrow M_0$ is a homomorphism of $(A,d)$-bimodules.

\begin{prop}\label{prop-2-diff-hom}
    Let $(A, d)$ be a difference algebra and $( (M_1 \xrightarrow{ \delta} M_0, \nu ), \Delta_0, \Delta_1, \theta)$ be a $2$-term bimodule up to homotopy over it. Then 
    \begin{align*}
        (\mathfrak{a}, \mathfrak{d}) = \big(   ( M_1 \xrightarrow{\widetilde{\delta } } A \oplus M_0 ~ \! , ~ \! \odot, \mu ), (d_0 = d \oplus \Delta_0 ~ \! , ~ \! d_1= \Delta_1 ~ \!, ~ \! d_2)   \big)
    \end{align*}
    is a $2$-term difference $A_\infty$-algebra, where the bilinear map $d_2 : (A \oplus M_0) \times (A \oplus M_0) \rightarrow M_1$ is given by
    \begin{align*}
        d_2 ((a, u), (b, v) ):= \theta (a, v) + \theta (u, b), \text{ for } (a, u), (b, v) \in A \oplus M_0.
    \end{align*}
\end{prop}

\begin{proof}
    Since $ ( M_1 \xrightarrow{\delta  }  M_0 , \nu )$ is a $2$-term bimodule up to homotopy of $A$, it follows from Theorem \ref{thm-2-hom} that $\mathfrak{a} =  ( M_1 \xrightarrow{\widetilde{\delta } } A \oplus M_0 ~ \! , ~ \! \odot, \mu )$ is a $2$-term $A_\infty$-algebra. It remains to show that $\mathfrak{d} = (d_0 = d + \Delta_0 ~ \!, ~ \! d_1 = \Delta_1 ~ \!, ~ \! d_2)$ is a difference operator on $\mathfrak{a}$. For any $(a, u), (b, v) \in A \oplus M_0$ and $\xi \in M_1$, we observe that
    \begin{align*}
        &d_0 (a, u) \odot (b, v) + (a, u) \odot d_0 (b, v) + d_0 (a, u) \odot d_0 (b, v) - d_0 ((a, u) \odot (b, v) ) \\
        &= \big( 0 ~ \! , ~ \! d(a) v + a \Delta_0 (v) + d(a) \Delta_0 (v) - \Delta_0 (a v) + \Delta_0 (u) b + u d(b) + \Delta_0 (u) d(b) - \Delta_0 (ub) \big) \\
        &= \big(  0 ~ \!, ~ \! \delta ( \theta (a, v) + \theta ( u, b) )  \big) = \widetilde{\delta} ( \theta (a, v) + \theta (u, b) ) = \widetilde{\delta} \big(  d_2 ((a, u), (b, v) )  \big)
    \end{align*}
    and
    \begin{align*}
        &d_0 (a, u) \odot \xi + (a, u) \odot d_1 (\xi) + d_0 (a, u) \odot d_1 (\xi) - d_1 ((a, u) \odot \xi ) \\
        &= d(a) \xi + a \Delta_1 (\xi) + d(a) \Delta_1 (\xi) - \Delta_1 (a \xi) \\
        &= \theta (a , \delta (\xi)) = d_2 ((a, u), (0, \delta (\xi)) ) = d_2 ((a, u), \widetilde{\delta} (\xi) ).
    \end{align*}
    Similarly, the condition (\ref{condi4}) can be used to show that
    \begin{align*}
        d_1 (\xi) \odot (a, u) + \xi \odot d_0 (a, u) + d_1 (\xi) \odot d_0 (a, u) - d_1 (\xi \odot (a, u) ) = d_2 ( \widetilde{\delta} (\xi), (a, u)).
    \end{align*}
    This verifies the identities (\ref{d1})-(\ref{d3}). Finally, by using (\ref{condi5})-(\ref{condi7}), one may also verify the identity (\ref{d4}). Hence $\mathfrak{d} = (d_0 = d + \Delta_0 ~ \! , ~ \! d_1= \Delta_1 ~ \!, ~ \! d_2)$ is a difference operator on $\mathfrak{a}$, and concludes the proof.
\end{proof}

The $2$-term difference $A_\infty$-algebra obtained in the above proposition is called the semidirect product of the difference algebra $(A, d)$ with the given $2$-term bimodule up to homotopy over it. Then it is easy to see that if the given $2$-term bimodule up to homotopy $((M_1 \xrightarrow{\delta} M_0, \nu), \Delta_0, \Delta_1, \theta)$ is skeletal (resp. strict) then the corresponding semidirect product as a $2$-term difference $A_\infty$-algebra is also skeletal (resp. strict).

\medskip

Combining Proposition \ref{prop-2-diff-hom} with the construction described by the functor ${\sf T}$ (see in the proof of Theorem \ref{thm-main}), we get the following result.

\begin{prop}
    Let $(A, d)$ be a difference algebra and $((M_1 \xrightarrow{\delta} M_0, \nu), \Delta_0, \Delta_1, \theta)$ be a $2$-term bimodule up to homotopy over it. Then 
    \begin{align*}
        \big( ( C = (A \oplus M_0\oplus M_1 \rightrightarrows A \oplus M_0), \bullet, \mathcal{A}) , D, \mathcal{D} \big)
    \end{align*}
    is a difference associative $2$-algebra, where the structure operations are given by
\begin{align*}
&\qquad s (a, u, \xi) = (a, u), \qquad t (a, u, \xi) = (a, u + \delta (\xi)), \qquad i_{(x, u)} = (x, u, 0), \\
& \qquad \qquad (a, u, \xi) \circ (b, v, \eta) = (a, u, \xi + \eta), \text{ when } t (a, u, \xi) = s (b, v, \eta),\\
& \qquad \qquad \qquad \qquad (a, u, \xi) \bullet (b, v, \eta) = (ab ~ \! , ~ \! a v + u b ~ \! , ~ \!  a \eta + \xi b), \\
& \mathcal{A}_{(a, u), (b, v), (c, w)} = ( ab c ~ \! , ~ \! (ab) w + (a v) c + (u b) c ~ \! , ~ \! \nu (a, b, w) + \nu (a, v, c) + \nu (u, b, c)  ),\\
& \qquad \qquad D_0 (a, u) = (d(a), \Delta_0 (u)), \qquad D_1 (a, u, \xi) = (d(a), \Delta_0 (u), \Delta_1 (\xi)), \\
&  \qquad \qquad \qquad \qquad \mathcal{D}_{(a, u), (b, v)} = \big( ab ~ \! , ~ \! a v + ub ~ \! , ~ \!  \theta (a, v) + \theta (u, b) \big), 
\end{align*}
for $(a, u), (b, v), (c, w) \in A \oplus M_0$ and $\xi, \eta \in M_1$.
\end{prop}

\medskip
    
\noindent  {\bf Acknowledgements.} The author would like to thank the Department of Mathematics, IIT Kharagpur, for providing a beautiful academic atmosphere in which the research was carried out.

\medskip

\noindent {\bf Data Availability Statement.} Data sharing does not apply to this article as no new data were created or analyzed in this study.


\begin{thebibliography}{BFGM03}

\bibitem{abad-crainic} C. A. Abad and M. Crainic, Representations up to homotopy of Lie algebroids, {\em J. Reine Angew. Math.} 663 (2012), 91–126.


   \bibitem{baez-crans} J. C. Baez and A. S. Crans, Higher-dimensional algebra VI: Lie $2$-algebras, {\em Theor. Appl.Categ.} 12 (2004), 492-538.





\bibitem{board} J. M. Boardman and R. M. Vogt, Homotopy invariant algebraic structures on topological spaces, Lecture Notes in Math., Vol. 347, Springer-Verlag, Berlin-New York, 1973.



 \bibitem{bott} R. Bott and L. W. Tu, Differential forms in algebraic topology, Grad. Texts in Math., Springer-Verlag, New York-Berlin, 1982.



 \bibitem{das-hom} A. Das, Hom-associative algebras up to homotopy, {\em J. Algebra} 556 (2020), 836–878.


  \bibitem{das} A. Das, Cohomology and deformations of crossed homomorphisms, {\em Bull. Belgian Math. Soc. Simon Stevin} Vol. 28, Issue 3 (2022), 381-397. 



 \bibitem{das-mandal} A. Das and A. Mandal, Extension, deformation and categorification of AssDer pairs, {\em Theor. Appl. Categ.} 45 (2026), no. 1, pp 1-32.

 \bibitem{das-ram} A. Das and R. Mandal, Quasi-twilled associative algebras, deformation maps and their governing algebras, to appear in {\em Homology Homotopy Appl.}, available at arXiv:2409.00443.




   \bibitem{getzler} E. Getzler, Lie theory for nilpotent $L_\infty$ algebras, {\em Ann. of Math.} 170 (2009), 271-301.



  \bibitem{guo-li} L. Guo, Y. Li, Y. Sheng and G. Zhou, Cohomologies, extensions and deformations of differential algebras of arbitrary weight, {\em Theor. Appl. Categ.} 38 (2022), no. 37, pp 1409-1433.




\bibitem{jia} J. Jiang and Y. Sheng, Deformations, cohomologies and integrations of relative difference Lie algebras, {\em J. Algebra} 614 (2023), 535-563.


\bibitem{keller} B. Keller, Introduction to $A_\infty$-algebras and modules, {\em Homology Homotopy Appl.} 3 (2001), no. 1, pp. 1-35.

\bibitem{kh} E. Khmaladze,  On associative and Lie $2$-algebras, {\em Proc. A. Razmadze Math. Inst.} 159 (2012), 57–64.

\bibitem{kolchin} E. R. Kolchin, Differential algebras and algebraic groups, Academic Press, New York, 1973.


   \bibitem{lada-markl} T. Lada and M. Markl, Strongly homotopy Lie algebras, {\em Comm. Algebra} 23 (1995), 2147-2161.

   \bibitem{lada-s} T. Lada and J. Stasheff, Introduction to SH Lie algebras for physicists, {\em Int. J. Theor. Phys.} 32 (1993), 1087-1103.

   \bibitem{laza} A. Lazarev and M. Movshev, On the cohomology and deformations of differential graded algebras, {\em J. Pure Appl. Algebra} 106 (1996), 141-151.

   \bibitem{loday} J.-L. Loday, On the operad of associative algebras with derivation, {\em Georgian Math. J.} 17 (2010), 347–372.

   \bibitem{lyu} W. Lyu, Z. Qi, J. Yang and G. Zhou, Formal deformations, cohomology theory and $L_\infty[1]$-structures for differential Lie algebras of arbitrary weight, {\em J. Geom. Phys.} 205 (2024), 105308.

   \bibitem{magid} A. R. Magid, Lectures on differential Galois theory, University Lecture Series 7, Amer. Math. Soc., 1994.

\bibitem{marco} M. Manetti, Lie methods in deformation theory, Springer Monogr. Math. Singapore, 2022.




 \bibitem{ritt} J. F. Ritt, Differential equations from the algebraic standpoint, {\em Amer. Math. Soc. Colloq. Publ., 14} (1934), Amer. Math. Soc., New York.




\bibitem{sheng-zhu} Y. Sheng and C. Zhu, Semidirect products of representations up to homotopy, {\em Pacific J. Math.} 249 (2011), no. 1, 211–236.

\bibitem{stas} J. Stasheff, Homotopy associativity of $H$-spaces. I, II, {\em Trans. Amer. Math. Soc.} 108 (1963), 293-312.



 \bibitem{voro} Th. Voronov, Higher derived brackets and homotopy algebras, {\em J. Pure Appl. Algebra} 202 (2005), 133--153.

   \bibitem{zhang-liu} S. Zhang and J. Liu, On Rota-Baxter Lie $2$-algebras, {\em Theor. Appl. Categ.} 39 (2023), no. 19, pp 545-566.

\end{thebibliography}
\end{document}